\documentclass[12pt]{amsart}
\usepackage{amsmath,amssymb}
\usepackage{hyperref}
\usepackage{datetime}

\DeclareMathOperator{\dd}{d}
\DeclareMathOperator{\Interior}{Interior}
\DeclareMathOperator{\lct}{lct}
\DeclareMathOperator{\LCT}{LCT}

\DeclareMathOperator{\OO}{\mathcal{O}}
\DeclareMathOperator{\row}{row}
\DeclareMathOperator{\Spec}{Spec}

\newtheorem{Theorem}{Theorem}[section]
\newtheorem{Lemma}[Theorem]{Lemma}
\newtheorem{Proposition}[Theorem]{Proposition}
\newtheorem{Corollary}[Theorem]{Corollary}

{\theoremstyle{definition}
\newtheorem{Definition}[Theorem]{Definition}

\newtheorem{Example}[Theorem]{Example}
\newtheorem{Remark}[Theorem]{Remark}
\newtheorem{Nada}[Theorem]{}
}
\numberwithin{equation}{Theorem}

\title[A procedure for computing the lct of a binomial ideal]{A procedure for computing the log canonical threshold of a binomial ideal}

\author{R. Blanco}
\address{Depto. Matem\'aticas,
Universidad de Castilla la Mancha. Spain.}
\email{MariaRocio.Blanco@uclm.es}

\author{S. Encinas}
\address{Depto. Matem\'atica Aplicada
and IMUVA, Instituto de Matem\'aticas.
Universidad de Valladolid.
Spain.}
\email{sencinas@maf.uva.es}

\thanks{The authors were partially supported by MTM2012-35849 and MTM2015-68524-P}

\begin{document}

\begin{abstract}
We present a procedure for computing the log-canonical threshold of an arbitrary ideal generated by binomials and monomials.

The computation of the log canonical threshold is reduced to the problem of computing the minimum of a function, which is defined in terms of the generators of the ideal.
The minimum of this function is attained at some ray of a fan which only depends on the exponents of the monomials appearing in the generators of the ideal.
\end{abstract}

\maketitle


\section*{Introduction}

The multiplier ideal of an ideal  $\mathfrak{a}$ can be defined from the analytic or algebraic
point of view, see for example \cite{BlickleLazarsfeld2004} or \cite{Lazarsfeld2010}. In this paper we use the algebraic approach which involves resolution of singularities.

Fix a log-resolution  $\Pi: Y\rightarrow X$  of an ideal $\mathfrak{a}\subseteq\mathcal{O}_X$ over a field of characteristic zero, the \emph{multiplier ideal} of $\mathfrak{a}$ is $\mathcal{J}(\mathfrak{a})=\Pi_{*}\mathcal{O}_Y(K_{Y/X}-F)$ where $K_{Y/X}$ is the relative canonical divisor, and the divisor $F$ defines the total transform of the ideal, $\mathfrak{a}\cdot \mathcal{O}_{Y}=\mathcal{O}_{Y}(-F)$. This definition can be extended for any real number $t\geq 0$, then we can attach to the ideal $\mathfrak{a}$ a collection of \emph{multiplier ideals} $\mathcal{J}(t\cdot\mathfrak{a})=\mathcal{J}(\mathfrak{a}^t)$.
These ideals and the invariants arising from them have been widely studied.
See \cite{BlickleLazarsfeld2004} for an introduction and \cite{EinLazarsfeldSmithVarolin2004} for some applications.

One of the main invariants defined in terms of multiplier ideals is the log canonical threshold.
The log canonical threshold of the ideal $\mathfrak{a}$ is the smallest number $t>0$ making the ideal $\mathcal{J}(\mathfrak{a}^t)$ non trivial, and it is a measure of the singularities of the functions $f\in \mathfrak{a}$.

Computing multiplier ideals and log canonical thresholds from their definition is difficult in general. In the case of monomial ideals, Howald \cite{Howald2001} proved that it is possible to compute the multiplier ideal and the log canonical threshold using the Newton polyhedron associated to the ideal.
For binomial ideals, some cases are known. Shibuta and Takagi \cite{ShibutaTakagi2009} gave a procedure based in linear programming to compute the log canonical threshold of complete intersection binomial ideals and the defining ideals of monomial curves in 3-dimensional space.
See also \cite{Thompson2014} where a formula for multiplier ideals of monomial curves in 3-dimensional space is presented.
\medskip

We say that an ideal $\mathfrak{a}$ is a m-binomial ideal if $\mathfrak{a}$ may be generated by monomials and binomials (\ref{DefGBinId}).
We prove in Theorem~\ref{ThResolLCT} that a log-resolution of $\mathfrak{a}$ is non necessary to compute the log canonical threshold of $\mathfrak{a}$.
It is enough to achieve what we have called a \emph{pseudo-resolution} (\ref{DefPseudoRes}), where the total transforms of the generators of the ideal are products of monomials and binomial hyperbolic equations. The key point is that, after a pseudo-resolution, computation of the log canonical threshold reduces to two simple cases addressed in Proposition~\ref{PropLctAfin}.
Weak pseudo resolutions (\ref{Defweak}) are close to the toric desingularization morphisms defined in \cite{GonzalezTeissier2002}, see also \cite[Section 6]{Teissier2004}.

Our main results are Theorems~\ref{ThlctMinv} and \ref{ProplctRay}.
Theorem~\ref{ThlctMinv} shows that computation of the log canonical threshold of a m-binomial ideal is reduced to the problem of computing the minimum of a function $\LCT(M^{+},M^{-},u):\mathbb{R}^n_{\geq 0}\to\mathbb{R}$ (\ref{DefLctM}), defined in terms of the generators of the ideal.
Let $\mathfrak{a}\subset k[x_1,\ldots,x_n]$ be a m-binomial ideal (\ref{DefGBinId}) generated by binomias and monomials $f_1,\ldots,f_r$.
The entries of the vector $u\in k^n$ are the coefficients involved in $f_1,\ldots,f_r$ and the rows of $M^{+}$ and $M^{-}$ encode the monomials appearing in the generators.

\setcounter{section}{5}
\setcounter{Theorem}{9}
\begin{Theorem} 


The $lct$ of the ideal $\mathfrak{a}$ is the minimum of the function 
$\LCT(M^{+},M^{-},u)$:
\begin{equation*}
\lct(W,\mathfrak{a})=\min\left\{\LCT(M^{+},M^{-},u)(v) \mid v\in\mathbb{R}^n_{\geq 0}\right\}.
\end{equation*}
\end{Theorem}

The function $\LCT(M^{+},M^{-},u)$ is defined for every $v\in\mathbb{R}^{n}_{\geq 0}$ and it is not continuous in general, however Proposition~\ref{PropStrata} shows that there exist a rational polyhedral fan $\Gamma$ with support $\mathbb{R}^{n}_{\geq 0}$ and such that the function $\LCT(M^{+},M^{-},u)$ is continuous in the relative interior of every cone of $\Gamma$.
The minimum of $\LCT(M^{+},M^{-},u)$, and hence the log canonical threshold of $\mathfrak{a}$, is attained at some ray of the fan $\Gamma$.

\begin{Theorem} 
The log-canonical threshold $\lct(W,\mathfrak{a})$ is the minimum of the values $\LCT(M^{+},M^{-},u,c)(v)$ where $v$ is a ray of the fan $\Gamma$.
\end{Theorem}

See complete version of Theorems~\ref{ThlctMinv} and \ref{ProplctRay} in section~\ref{SecMain} for all technical details.
\setcounter{section}{0}
\setcounter{Theorem}{0}
\medskip

Our results generalize the procedure presented in
\cite{ShibutaTakagi2009} and allow us to calculate the log canonical threshold for arbitrary binomial ideals, including the non complete intersection case.
The procedure of Shibuta and Takagi relies on a linear programming problem
formulated only in terms of the exponents of the monomials appearing in
the generators of the ideal.
We illustrate in Example~\ref{ExDepCoef} that the log canonical threshold also depends on the coefficients of the binomials generating the ideal.
\medskip

We also show, see Corollary~\ref{CorPseudo}, a constructive procedure to obtain a pseudo-resolution of a 
m-binomial ideal. This procedure is based on Zeillinger's idea \cite{Zeillinger2006} for solving Hironaka's polyhedra game. Using this idea, all blowing-up centers are of codimension two. We use the same invariants as in \cite{Goward2005} where the author obtains a log-resolution for monomial ideals.
\medskip

We include in section~\ref{SecExample} several examples to illustrate our method. All computations were made with Singular \cite{Singular2012}.

The authors want to thank to the editors and the anonymous reviewers/referees for valuable suggestions and comments.

\section{Log-resolution}

In what follows $k$ is a field of characteristic zero.
We denote $W$ to be a smooth algebraic variety over $k$ and we recall definitions of log-resolution, multiplier ideal and log canonical threshold.

\begin{Definition} \label{logr}
Let $\mathfrak{a}\subset \mathcal{O}_W$ be a non zero sheaf of ideals on $W$.
A \emph{log-resolution} of $\mathfrak{a}$ is a proper birational morphism $\Pi:W'\to W$ such that
\begin{itemize}
\item $W'$ is smooth over $k$,

\item the total transform of the ideal $\mathfrak{a}\OO_{W'}=\OO_{W'}(-F)$ is an invertible sheaf associated to a normal crossing divisor $F$ in $W'$,

\item and $\mathrm{Exc}(\Pi)\cup\mathrm{Supp}(F)$ is a simple normal crossing divisor, where $\mathrm{Exc}(\Pi)$ is the exceptional locus of $\Pi$.
\end{itemize}
\end{Definition}

It is well known that if the field $k$ has characteristic zero, then log-resolution of ideals exists.
In fact there are procedures to obtain the morphism $\Pi$ as a constructive sequence of blowing ups
\begin{equation*}
(W,E)=(W^{(0)},E^{(0)})\leftarrow (W^{(1)},E^{(1)})\leftarrow \cdots  \leftarrow (W^{(N)},E^{(N)}).
\end{equation*}
See \cite{EncinasVillamayor2000}, \cite{BravoEncinasVillamayor2005} for details or \cite{Hauser2003} for an extended review.

\begin{Definition}
Let $t\geq 0$ be a real number.
The multiplier ideal of $\mathfrak{a}\subset\OO_{W}$ with exponent $t$ is:
\begin{equation*}
    \mathcal{J}(W,\mathfrak{a}^{t})=
    \Pi_{\ast}\OO_{W'}(K_{W'/W}-\lfloor tF \rfloor),
\end{equation*}
where $\Pi:W'\to W$ is a log-resolution of $\mathfrak{a}$ and
$\mathfrak{a}\OO_{W'}=\OO_{W'}(-F)$.
\medskip

If $\Delta$ is an effective Cartier divisor in $W$, we also define the multiplier ideal associated to the ideal and the divisor:
\begin{equation*}
    \mathcal{J}(W,\Delta,\mathfrak{a}^{t})=
    \Pi_{\ast}\OO_{W'}(K_{W'/W}-\Pi^{\ast}\Delta-\lfloor tF \rfloor),
\end{equation*}
where here $\Pi:W'\to W$ is a log-resolution of $\mathfrak{a}$ with the 
additional condition that the divisor $F+\Pi^{\ast}\Delta$ has normal 
crossings.
\end{Definition}

The definition of the multiplier ideal $\mathcal{J}(W,\Delta,\mathfrak{a}^{t})$ is independent of the choice of the log-resolution and it can be generalized when $\Delta$ is a $\mathbb{Q}$-divisor and even to the case $W$ non-smooth, see \cite{Lazarsfeld2004_2}.

Note that the set of multiplier ideals $\{\mathcal{J}(W,\Delta,\mathfrak{a}^{t}) \mid t\geq 0\}$ is a filtration, $\mathcal{J}(W,\Delta,\mathfrak{a}^{t_1})\supseteq\mathcal{J}(W,\Delta,\mathfrak{a}^{t_2})$ for $t_1\leq t_2$.
The first real number where the multiplier ideal is non trivial is called the log canonical threshold.
\begin{Definition} \label{Deflct}
The log canonical threshold of an ideal $\mathfrak{a}\subset\OO_W$ is
\begin{equation*}
\lct(W,\mathfrak{a})=\inf\{t\mid \mathcal{J}(W,\mathfrak{a}^{t})\neq\OO_{W}\}.
\end{equation*}
If $\Delta$ is a Cartier divisor in $W$ then 
\begin{equation*}
\lct(W,\Delta,\mathfrak{a})=\inf\{t\mid \mathcal{J}(W,\Delta,\mathfrak{a}^{t})\neq\OO_{W}\}.
\end{equation*}
The definition may be local at a point $\xi\in W$,
\begin{equation*}
\lct(W,\Delta,\mathfrak{a})_{\xi}=
\inf\{t\mid \mathcal{J}(W,\Delta,\mathfrak{a}^{t})\neq\OO_{W,\xi}\}.
\end{equation*}
\end{Definition}
The log canonical threshold is a rational number, see \cite{Lazarsfeld2004_2}.

\begin{Proposition} \label{PropLctCover}
Let $\Pi:W'\to W$ be a proper birational morphism where $W'$ and $W$ are smooth varieties over $k$.
Let $U'_1\cup\cdots\cup U'_r=W'$ be an open cover in $W'$.
Let $\Delta$ be a Cartier divisor in $W$. If $\Delta'=\Pi^{\ast}\Delta-K_{W'/W}$ then
\begin{equation*}
\lct(W,\Delta,\mathfrak{a})=\min\{\lct(U'_i,\Delta',\mathfrak{a}^{\ast}) \mid i=1,\ldots,r\}.
\end{equation*}
\end{Proposition}

\begin{proof}
It follows from Definition~\ref{Deflct} and the fact that $\mathcal{J}(W,\Delta,\mathfrak{a})=\Pi_{\ast}\mathcal{J}(W',\Delta',\mathfrak{a}^{\ast})$.
\end{proof}

Note that the morphism $\Pi$ in Proposition \ref{PropLctCover} need not to be a log-resolution of $\mathfrak{a}$.  We will apply this result for $\Pi$ a pseudo resolution (\ref{DefPseudoRes}) of a m-binomial ideal.

\section{Toric varieties}

We remind here some basic notions about toric varieties, for more details see \cite{CoxLittleSchenck2011} \cite{Fulton1993} or \cite{Oda1988}. See also \cite{GonzalezTeissier2014} for generalization to non necessarily normal toric varieties.
\medskip

Fix $N\cong \mathbb{Z}^{n}$ a $n$-dimensional lattice and let $\mathcal{M}=Hom(N,\mathbb{Z})$ be its dual lattice. Denote by $N_\mathbb{R}$ the real vector space spanned by $N$. 

\begin{Definition}
A \emph{cone} $\sigma$ in $N_\mathbb{R}$ is a \emph{strongly convex rational polyhedral cone}, that is a set of non negative linear combinations of some vectors $v_{1},\ldots,v_{r}\in N$ such that it contains no nonzero $\mathbb{R}$-subspace of $N_\mathbb{R}$. 
\medskip

A \emph{face} of a cone $\sigma$ is a subset $\tau\subset\sigma$ such that there exists $w\in\mathcal{M}$ with
\begin{equation*}
\tau=\sigma\cap w^{\perp}=\{u\in\sigma \mid \langle w,u \rangle=0 \}.
\end{equation*}
One dimensional faces are called the rays of $\sigma$ and we denote $\sigma(1)$ the set of all rays of $\sigma$.
Note that if $\rho\in\sigma(1)$ then there exist a unique primitive vector $v_{\rho}\in N$ generating the semi-group $\rho\cap N$.

A \emph{fan} $\Sigma$ in $N$ is a set of strongly convex rational polyhedral cones $\sigma$ in $N_\mathbb{R}$ such that every face of a cone $\sigma\in\Sigma$ is also a cone in $\Sigma$ and the intersection of two cones in $\Sigma$ is a face of each one of them. 
\medskip

Given a fan $\Sigma$ in $N$, the \emph{support} of $\Sigma$, $|\Sigma|$, is the union of all the cones in $\Sigma$, that is, the set  $|\Sigma|=\cup_{\sigma\in\Sigma}\ \sigma\subset N_\mathbb{R}$.
We denote $\Sigma(1)$ the set of all rays in $\Sigma$. By abuse of notation we will also denote $\sigma(1)$, resp. $\Sigma(1)$, the set of primitive vectors $v_{\rho}$ with $\rho\in\sigma(1)$, resp. $\rho\in\Sigma(1)$. 
\medskip

If $\sigma$ is a cone in $N$, the \emph{dual cone} $\sigma^{\vee}\subset \mathcal{M}_{\mathbb{R}}$ is the set of vectors in $\mathcal{M}_{\mathbb{R}}$ that are nonnegative on $\sigma$, that is
\begin{equation*}
\sigma^{\vee}=\{w \in \mathcal{M}_{\mathbb{R}} | \langle w,u\rangle\geq 0\ \forall u\in \sigma \}.
\end{equation*}
\end{Definition}

\begin{Nada}
The semi-group $\sigma^{\vee}\cap \mathcal{M}=\{w \in \mathcal{M} | \langle w,u\rangle\geq 0\ \forall u\in \sigma \}$ is finitely generated.
Hence the algebra of the semi-group $k[\sigma^{\vee}\cap \mathcal{M}] $ is a finitely generated $k$-algebra that defines an affine toric variety $U_{\sigma}=Spec(k[\sigma^{\vee}\cap \mathcal{M}])$.
In fact every affine normal toric variety is of this form.
\medskip

Given a fan $\Sigma$ in $N$ we associate a (normal) toric variety $W_{\Sigma}$ obtained by gluing the affine toric varieties $\{W_{\sigma}\mid \sigma\in\Sigma\}$, see \cite{CoxLittleSchenck2011} for details.
\end{Nada}

\begin{Remark} \label{RemRegCone}
We say that a cone $\sigma$ is regular if the primitive vectors $\sigma(1)$ are part of a $\mathbb{Z}$-basis of the lattice $N$.
A fan $\Sigma$ is regular if every cone $\sigma\in\Sigma$ is regular.
It is known that the toric
variety $W_{\Sigma}$ is regular if and only if the associated fan $\Sigma$ is regular.

If $\sigma$ is a regular cone, there exist a $\mathbb{Z}$-basis of $N$, say $v_1,\ldots,v_n$ such that $\sigma(1)=\{v_1,\ldots,v_r\}$ for some $r\leq n$.
Let $w_1,\ldots,w_n$ be the dual basis in $\mathcal{M
}$, 
then the dual cone $\sigma^{\vee}$ is generated
by $w_1,\ldots,w_r,\pm w_{r+1},\ldots,\pm w_n$.
The associated affine toric variety is $U_{\sigma}=\Spec(k[x_1,\ldots,x_r,x_{r+1}^{\pm},\ldots,x_{n}^{\pm}])$.
\end{Remark}

\begin{Remark}
\label{RemComplex}
Let $\Sigma$ be a fan and $\Sigma(1)=\{v_1,\ldots,v_{m'}\}\subset N$ be the set of primitive vectors generating the rays of $\Sigma$.

Set $N'=\left\langle \Sigma(1)\right\rangle\cap N$.
Recall that the toric variety $W_{\Sigma}$ has a torus factor if and only if $N'_{\mathbb{R}}\neq N_{\mathbb{R}}$ (\cite[3.3.9]{CoxLittleSchenck2011}).
The quotient $N/N'$ is torsion-free and there exists a complement $N''\subset N$ such that $N=N'\oplus N''$.
\medskip

Note that every cone $\sigma\in\Sigma$ is generated by a subset of $\Sigma(1)=\{v_1,\ldots,v_{m'}\}$, $\sigma(1)=\{v_{i_1},\ldots,v_{i_r}\}$ for some indexes $i_1,\ldots,i_r$.

We associate to a fan $\Sigma$, and indeed to the toric variety $W_{\Sigma}$, 
a set of vertices $\Xi\subset N$.
\begin{equation*}
\Xi=\{v_1,\ldots,v_{m'},v_{m'+1},\ldots,v_m\},
\end{equation*}
where $\{v_{m'+1},\ldots,v_m\}$ is a $\mathbb{Z}$-basis of the complement $N''$.
\end{Remark}

\begin{Nada} \label{RemDescCone}
Note that for every cone $\sigma\in\Sigma$ there are vertices
$v_{i_1},\ldots,v_{i_n}\in\Xi$ such that:
\begin{itemize}
\item $v_{i_1},\ldots,v_{i_n}$ is a $\mathbb{Z}$-basis of $N$,

\item $\sigma(1)=\{v_{i_1},\ldots,v_{i_r}\}$ for some $r\leq n$.
\end{itemize}

Note that $\Xi$ is not unique since there are many choices for the vertices $v_{m'+1},\ldots,v_{m}$ as possible $\mathbb{Z}$-basis of $N''$.
These additional vertices are needed in order to define globally monomials and binomials in a smooth toric variety, see Definition \ref{DefGlobalMon}.
\end{Nada}

\begin{Remark} \label{RemTotalCRing}
Recall that the ring $k[v_1,\ldots,v_m]$ is the total coordinate ring of $W_{\Sigma}$. This ring has a grading given by the divisor class group $\mathrm{Cl}(W_{\Sigma})$. Two monomials $v^{a}$ and $v^{b}$ have the same degree if $a-b$ satisfies the same linear dependencies as the $v$'s:
If $\sum_{i=1}^{m}\lambda_i v_i=0$ then $\sum_{i=1}^{m}\lambda_i(a_i-b_i)=0$.
See \cite[\S 5.2]{CoxLittleSchenck2011}.
\end{Remark}

\begin{Definition} \label{DefSubdiv}
\cite[Def. 3.3.13]{CoxLittleSchenck2011}
Let $\Sigma$ be a fan in $N$. A fan $\Sigma'$ is a \emph{subdivision} of the fan $\Sigma$ if both fans have the same support and if every cone $\sigma'\in\Sigma'$ is contained in a cone $\sigma\in\Sigma$.

Let $\sigma$ be a regular cone generated by $v_1,\ldots,v_r$ and set $\Sigma$ the minimum fan containing $\sigma$. The cones in $\Sigma$ are $\sigma$ and all the faces of $\sigma$.
Fix a face $\tau$ of $\sigma$, and assume that $\tau$ is generated by $v_1,...,v_s$, $s\leq r$.
The \emph{star subdivision} of $\Sigma$ with center a face $\tau\in\Sigma$ is the fan $\Sigma'$ containing the cones $\sigma_1,\ldots,\sigma_s$ and all their faces, where each cone $\sigma_i$, $i=1,\ldots,s$ is generated by
\begin{equation} \label{eqConosExpl}
v_1,\ldots,v_{i-1},v_1+\cdots+v_s,v_{i+1},\ldots,v_s,
v_{s+1},\ldots,v_r.
\end{equation}

If $\Sigma$ is a fan and $\tau\in\Sigma$ is a cone, then the star subdivision of $\Sigma$ with center $\tau$ is the fan $\Sigma'$ such that:
\begin{itemize}
\item If $\sigma\in\Sigma$ and $\tau$ is not a face of $\sigma$ then $\sigma\in\Sigma'$.

\item If $\sigma\in\Sigma$ and $\tau$ is a face of $\sigma$ then $\sigma_i\in\Sigma'$ for $i=1,..s$, where $\sigma_1,\ldots,\sigma_s$ are cones as in (\ref{eqConosExpl}).
\end{itemize}
\end{Definition}

\begin{Nada}
Recall that any subdivision $\Sigma'$ of a fan $\Sigma$ defines a proper birational morphism $W_{\Sigma'}\to W_{\Sigma}$.
\medskip

If $\Sigma$ is a regular fan (i.e. $W_{\Sigma}$ is a regular toric variety) and $\Sigma'$ is a star subdivision of $\Sigma$, then $\Sigma'$ is also a regular fan.
\medskip

Let $\Sigma'$ be the star subdivision of $\Sigma$ with center $\tau$.
If $\Xi=\{v_1,\ldots,v_m\}$ is a set of vertices for $\Sigma$ then a set of vertices $\Xi'$ for $\Sigma'$ is obtained by adding one element.

If $\tau$ is generated by the vertices $v_{i_{1}},\ldots,v_{i_{r}}$ of $\Xi$ then
$\Xi'=\Xi\cup\{v_{i_{1}}+\cdots+v_{i_{r}}\}$.
\end{Nada}

\begin{Definition}
Let $W=W_{\Sigma}$ be a regular toric variety, defined by a regular fan $\Sigma$.
A \emph{combinatorial blowing up} with center a cone $\tau\in\Sigma$ is the morphism
\begin{equation*}
W'=W_{\Sigma'}\to W=W_{\Sigma}
\end{equation*}
defined by the star subdivision $\Sigma'$ of $\Sigma$ with center $\tau$ (\ref{DefSubdiv}).
\end{Definition}

\begin{Remark} \label{RemCombBlup}
Combinatorial blowing ups correspond to usual blowing ups with center defined by some variables.
\medskip

Let $\Pi:W_{\Sigma'}\to W_{\Sigma}$ be the combinatorial blow-up with center $\tau\in\Sigma$.

Fix a cone $\sigma\in\Sigma$ such that $\tau$ is a face of $\sigma$.
Set $v_1,\ldots,v_{r}\in N$ be generators of $\sigma$ such that $v_1,\ldots,v_s$, $s\leq r$ are generators of $\tau$. The open set $U_{\sigma}\subset W_{\Sigma}$ is, by Remark \ref{RemRegCone},
\begin{equation*}
U_{\sigma}\cong\Spec(k[x_1,\ldots,x_{r},x_{r+1}^{\pm},\ldots,x_n^{\pm}]).
\end{equation*}
The cone $\sigma$ is replaced in $\Sigma'$ by the $s$ cones $\sigma_1,\ldots,\sigma_s$, notation as in Definition \ref{DefSubdiv}.
The restriction of the combinatorial blow-up $\Pi$ to the open set $U_{\sigma}$,
\begin{equation*}
U_{\sigma_1}\cup\cdots\cup U_{\sigma_s}\to U_{\sigma},
\end{equation*}
is the (usual) blowing-up of $U_{\sigma}$ with center $(x_1,\ldots,x_s)$.

Extend the generators of $\sigma$ to $v_1,\ldots,v_n$ a $\mathbb{Z}$-basis of $N$ and denote $w_1,\ldots,w_n$ the dual basis.
The cone $\sigma^{\vee}$ is generated by $w_1,\ldots,w_r,\pm w_{r+1},\ldots,\pm w_n$.
It is easy to check that for $i=1,\ldots,s$ the cone $\sigma_i^{\vee}$ is generated by
\begin{equation*}
w_1-w_i,\ldots,w_{i-1}-w_i,w_i,w_{i+1}-w_i,\ldots,w_s-w_i,
w_{s+1},\ldots,w_r,\pm w_{r+1},\ldots,\pm w_n
\end{equation*}
and 
\begin{equation*}
U_{\sigma_i}\cong
\Spec\left( k\left[
\frac{x_1}{x_i},\ldots,\frac{x_{i-1}}{x_i},x_i,
\frac{x_{i+1}}{x_i},\ldots,\frac{x_s}{x_i},
x_{s+1},\ldots,x_r,x_{r+1}^{\pm},\ldots,x_n^{\pm}
\right]\right).
\end{equation*}
\end{Remark}

\begin{Definition} \label{DefGlobalMon}
Let $W=W_{\Sigma}$ be a smooth toric variety. Let $\Xi=\{v_1,\ldots,v_m\}$ be a set of vertices for $\Sigma$.

A \emph{monomial} in $W$ will be a monomial in the total coordinate ring (\ref{RemComplex}). We will denote by $v^a=v_1^{a_1}\cdots v_m^{a_m}$ where $a=(a_{1},\ldots,a_{m})\in\mathbb{Z}^m_{\geq 0}$.

If $\sigma\in\Sigma$ is the cone generated by $v_{i_1},\ldots,v_{i_r}$, recall that the corresponding open set in $W$ is
\begin{equation*}
U_{\sigma}\cong \Spec(k[x_1,\ldots,x_r,x_{r+1}^{\pm},\ldots,x_n^{\pm}])
\end{equation*}
where $v_{i_1},\ldots,v_{i_r},v_{i_{r+1}},\ldots,v_{i_n}$ is a $\mathbb{Z}$-basis of $N$ (\ref{RemDescCone}).
The monomial $v^{a}$ induces in the open set $U_{\sigma}$ the monomial $x_1^{a_{i_1}}\cdots x_n^{a_{i_n}}$.

It follows that the monomial $v^{a}$ defines a sheaf of ideals in $\OO_{W}$ which is locally monomial.

A \emph{binomial} in $W$ will be a homogeneous binomial in the total coordinate ring (\ref{RemTotalCRing}).
We will denote $v^a-uv^b$, where $u\in k^{\ast}$, $a,b\in\mathbb{Z}^m$, such that
\begin{itemize}
\item $v^a$ and $v^b$ are monomials, i.e. for every cone $\sigma$, with $\sigma(1)=\{v_{i_1},\ldots,v_{i_r}\}\subset\Xi$ then
\begin{equation*}
a_{i_j}\geq 0 \qquad\text{and}\qquad b_{i_j}\geq 0
\qquad\text{for}\qquad j=1,\ldots,r.
\end{equation*}

\item And if 
\begin{equation*}
\sum_{i=1}^{m}\lambda_i v_i=0 
\qquad \Longrightarrow \qquad
\sum_{i=1}^{m}\lambda_i(a_i-b_i)=0.
\end{equation*}
\end{itemize}

An irreducible binomial is a binomial as above of the form $v^{\beta^{+}}-uv^{\beta^{-}}$, where $\beta\in\mathbb{Z}^m$ and for $i=1,\ldots,m$
\begin{equation*}
\beta^{+}_i=\left\{
\begin{array}{lll}
\beta_i & \text{if} & \beta_i>0 \\ 
0 & \text{if} & \beta_i\leq 0
\end{array} 
 \right.,
\qquad
\beta^{-}_i=\left\{
\begin{array}{lll}
-\beta_i & \text{if} & \beta_i<0 \\ 
0 & \text{if} & \beta_i\geq 0
\end{array} 
 \right..
\end{equation*}
\end{Definition}

See \cite[\S 5.3]{CoxLittleSchenck2011} for more details for sheaves on toric varieties.
A binomial defines a hypersurface in $W_{\Sigma}$ such that at every affine open set $U_{\sigma}\subset W_{\Sigma}$ it is defined by a usual binomial.

\begin{Definition} \label{DefGBinId}
A \emph{monomial ideal} in $W=W_{\Sigma}$ is the ideal sheaf in $\OO_W$ generated by a finite set of monomials $v^{a_1},\ldots,v^{a_{\ell}}$. At every open subset $U_{\sigma}$, $\sigma\in\Sigma$ corresponds to an ideal generated by monomials in a polynomial ring.

A \emph{binomial ideal} in $W_{\Sigma}$ is the ideal sheaf generated by a finite set of binomials.

A \emph{m-binomial ideal} in $W_{\Sigma}$ is the ideal sheaf generated by a finite set of monomials and binomials.
\end{Definition}

\begin{Proposition} \label{PropIrredBin}
Every binomial $v^a-uv^b$ may be expressed as a product of a monomial and an irreducible binomial
\begin{equation*}
v^a-uv^b=v^{\alpha}(v^{\beta^{+}}-uv^{\beta^{-}}).
\end{equation*}
\end{Proposition}

\begin{proof}
Set $\alpha_i=\min\{a_i,b_i\}$, for $i=1,\ldots,m$, and $\beta=a-b$. We factor out as much as possible in $v^{\alpha}$, then some components of $\beta$ will be positive and some will be negative, getting an irreducible binomial $v^{\beta^{+}}-uv^{\beta^{-}}$ like in Definition \ref{DefGlobalMon},
note that $a-b=\beta^{+}-\beta^{-}$. 

\end{proof}

\section{Pseudo-resolution}

Combinatorial blowing up are enough to obtain a log-resolution of monomial ideals, but it is not possible to obtain log-resolution of binomial ideals only with this transformations.
However we will prove that these combinatorial transformations are enough to obtain pseudo-resolutions which allow to compute the log canonical threshold.
\medskip

In what follows $\Sigma$ is always a regular fan and $\Xi=\{v_1,\ldots,v_m\}$ is a set of vertices of $\Sigma$.

\begin{Proposition} \label{PropTTransMon}
Let $\Sigma'$ be the star subdivision of $\Sigma$ with center $\sigma\in\Sigma$, let $\Xi=\{v_1,\ldots,v_m\}$ be a set of vertices of $\Sigma$ and let $\Xi'=\{v_1,\ldots,v_m,v_{m+1}\}$ be a set of vertices of $\Sigma'$.
Assume $\sigma(1)=\{v_{i_{1}},\ldots,v_{i_{r}}\}$ are the vertices of $\sigma$.

If $v^a=v_1^{a_1}\cdots v_m^{a_m}$ is a monomial in $W_{\Sigma}$, then the
total transform of $v^a$ is the monomial in $W_{\Sigma'}$
\begin{equation*}
v_1^{a_1}\cdots v_m^{a_m} v_{m+1}^{a_{m+1}},
\qquad
a_{m+1}=a_{i_{1}}+\cdots+a_{i_{r}}.
\end{equation*}

If $f=v^{\beta^{+}}-uv^{\beta^{-}}$, with $\beta=(\beta_1,\ldots,\beta_m)\in\mathbb{Z}^m$, is an irreducible binomial in $W_{\Sigma}$, then the strict 
transform of $f$ in $W_{\Sigma'}$ is $f'=v^{{\beta'}^{+}}-uv^{{\beta'}^{-}}$ where
\begin{equation*}
\beta'=(\beta_1,\ldots,\beta_m,\beta_{m+1}), \qquad
\beta_{m+1}=\beta_{i_{1}}+\cdots+\beta_{i_{r}}
\end{equation*}
\end{Proposition}

\begin{proof}
It follows from Remark \ref{RemCombBlup} and Proposition \ref{PropIrredBin}.
\end{proof}

Note that if $\mathfrak{a}$ is a monomial (resp. m-binomial) ideal then after a
combinatorial blowing up $W'\to W$ the total transform of $\mathfrak{a}$ is 
again a monomial (resp. m-binomial) ideal.
In fact this is also true for any regular subdivision $\Sigma'$ of $\Sigma$, even if $\Sigma'$ is not obtained by a sequence of star subdivisions.
\medskip

If $a=(a_1,\ldots,a_m)\in\mathbb{Z}^m$ and $\sigma\in\Sigma$ is a cone with $\sigma(1)=\{v_{i_1},\ldots,v_{i_r}\}$ then we denote $a_{\sigma}=(a_{i_1},\ldots,a_{i_r})\in\mathbb{Z}^r$.
\begin{Definition} \label{DefHiperb}
We say that an irreducible binomial $f=v^{\beta^{+}}-uv^{\beta^{-}}$, with $\beta\in\mathbb{Z}^m$, is a \emph{hyperbolic equation} if for every cone $\sigma\in\Sigma$
then either $\beta^{+}_{\sigma}=0$ or $\beta^{-}_{\sigma}=0$.
\medskip

A binomial $f=v^a-uv^b$ in $W_{\Sigma}$ is said to be \emph{weakly resolved} if $f$ may be expressed as the product of a monomial and a hyperbolic equation.
Using the notation of Proposition \ref{PropIrredBin} this means that if $\beta=a-b$ then for every cone $\sigma\in\Sigma$ either $\beta^{+}_{\sigma}=0$ or
$\beta^{-}_{\sigma}=0$.
\medskip

We will say that every monomial $v^a$ is weakly resolved.
\end{Definition}

Let $\Sigma'$ be a subdivision of $\Sigma$ obtained by a sequence of star subdivisions, and let $f$ be a binomial in $W_{\Sigma}$.
Note that the morphism $W_{\Sigma'}\to W_{\Sigma}$ is a log-resolution of the ideal generated by $f$, if and only if the total transform of $f$ is weakly resolved.

For an ideal generated by more that one binomial the above result is not true, in this case one can obtain an embedded desingularization of the variety defined by the binomials, see \cite[Section 6]{Teissier2004}.

\begin{Definition} \label{Defweak}
Let $\mathfrak{a}\subset\OO_{W}$ be a 
m-binomial ideal generated by binomials and monomials $f_{1},\ldots,f_{r}$.

A \emph{weak-pseudo-resolution} of $\mathfrak{a}$ with respect to the generators $f_1,\ldots,f_r$ is a sequence of combinatorial blowing-ups
$W'_{\Sigma'}\to W_{\Sigma}$ such that the total transforms of generators $f^{\ast}_{1},\ldots,f^{\ast}_{r}$
are weakly-resolved as in (\ref{DefHiperb}).
\end{Definition}

The above definition depends on the generators of the ideal.
Note also that every monomial ideal is already weak-pseudo-resolved.

\begin{Definition} \label{DefParcOrd}
We denote the partial ordering $\preceq$ in $\mathbb{Z}^{m}$, defined 
as follows:
If $\alpha,\beta\in\mathbb{Z}^{m}$ then $\alpha\preceq\beta$ if and 
only if $\alpha_{i}\leq\beta_{i}$ for all $i=1,\ldots,m$.
Where 
$\alpha=(\alpha_{1},\ldots,\alpha_{m})$ and 
$\beta=(\beta_{1},\ldots,\beta_{m})$.

It is the usual partial ordering given by division of monomials.
\end{Definition}
Note that if a binomial $v^{a}-uv^{b}$ in $W_{\Sigma}$ is weakly resolved then for every $\sigma\in\Sigma$ we have that either $a_{\sigma}\preceq b_{\sigma}$ or $b_{\sigma}\preceq a_{\sigma}$.

Using the above partial ordering we can refine a weak-pseudo-resolution to a pseudo-resolution.
\begin{Definition} \label{DefPseudoRes}
Let $\mathfrak{a}\subset\OO_{W}$ be a 
m-binomial ideal generated by binomials and monomials $f_{1},\ldots,f_{r}$.

A \emph{pseudo-resolution} of $\mathfrak{a}$ with respect to generators $f_1,\ldots,f_r$ is a weak-pseudo-resolution
$W'\to W$ such that if $f^{\ast}_{1},\ldots,f^{\ast}_{r}$ are the total transforms
and for $i=1,\ldots,r$
\begin{align*}
f^{\ast}_{i}=v^{\alpha_{i}}(v^{\beta^{+}_{i}}-u_i v^{\beta^{-}_{i}}) & 
\text{ if } f_i \text{ is a binomial, or} \\
f^{\ast}_{i}=v^{\alpha_{i}} & 
\text{ if } f_i \text{ is a monomial},
\end{align*}
then for every cone $\sigma\in\Sigma'$ the vectors
$(\alpha_{1})_{\sigma},\ldots,(\alpha_{r})_{\sigma}$
are totally ordered with respect to the ordering $\preceq$.
\end{Definition}

Note that in order to prove if a weak-pseudo-resolution is a pseudo-resolution it is enough to check 
the condition of totally ordered only for maximal cones of $\Sigma'$.

\begin{Proposition} \label{PseudoExist}
Let $\mathfrak{a}$ be a m-binomial ideal generated by $f_{1},\ldots,f_{r}$.
There exists a pseudo-resolution of the ideal $\mathfrak{a}$ and it can be obtained by a sequence of star subdivisions of $\Sigma$.
\end{Proposition}

\begin{proof}
First one may produce log-resolutions of every generator $f_{i}$ 
to obtain a weak-pseudo-resolution $W'\to W$ of the ideal $\mathfrak{a}$. The total 
transforms $f^{\ast}_{1},\ldots,f^{\ast}_{r}$ can be  expressed as
\begin{equation*}
    f^{\ast}_{i}=v^{\alpha_{i}}(v^{\beta_i^{+}}-u_i v^{\beta_{i}^{-}}),
    \qquad\text{or}\qquad
    f^{\ast}_{i}=v^{\alpha_{i}},
\end{equation*}
where each $v^{\beta_i^{+}}-u_i v^{\beta_{i}^{-}}$, $i=1,\ldots,r$ is a hyperbolic equation.

Now consider the monomial ideals generated by the pairs of monomials
\begin{equation*}
    \{v^{\alpha_{i}},v^{\alpha_{j}}\},\qquad \forall i<j,
\end{equation*}
and construct a simultaneous log-resolution for all these monomial ideals, say $W''\to W'$. 
For every $\sigma\in\Sigma''$ we have that either $(\alpha_i)_{\sigma}\preceq(\alpha_j)_{\sigma}$ or 
$(\alpha_j)_{\sigma}\preceq(\alpha_i)_{\sigma}$, which means that the set
$\{(\alpha_{1})_{\sigma},\ldots,(\alpha_{r})_{\sigma}\}$ is totally ordered with respect to the ordering $\preceq$. We conclude that $W''\to W$ is a pseudo-resolution of the ideal generated by 
$f_{1},\ldots,f_{r}$.
\end{proof}

\begin{Remark}
Suppose $W'\to W$ is a pseudo resolution of $\mathfrak{a}=\left\langle f_1,\ldots,f_r\right\rangle$ (\ref{DefPseudoRes}).
Let $\sigma\in\Sigma'$ be a cone with 
$\sigma(1)=\{v_{i_1},\ldots,v_{i_s}\}$, so that the affine open set $U_{\sigma}\subset W'$ is
\begin{equation*}
U_{\sigma}=\Spec\left(k[x_1\ldots,x_s,x_{s+1}^{\pm},\ldots,x_n^{\pm}]\right)
\end{equation*}
The total transform of each $f_i$ in $U_{\sigma}$ is,
\begin{itemize}
\item $\left(f^{\ast}_i\right)|_{U_{\sigma}}=\tilde{w}_i x^{\tilde{\alpha}_i}(1-\tilde{u}_i x^{\tilde{\beta}_i})$,
if $f_i$ is a binomial or
\item $\left(f^{\ast}_i\right)|_{U_{\sigma}}=x^{\tilde{\alpha}_i}$, if $f_i$ is a monomial.
\end{itemize}
where $\tilde{\alpha}_i=(\alpha_i)_{\sigma}$ and
\begin{itemize}
\item $\tilde{\beta}_i=(\beta_i^{-}$, $\tilde{u}_i=u_i$, $\tilde{w}_i=1$, if $(\beta_i^{+})_{\sigma}=0$ or
\item $\tilde{\beta}_i=(\beta_i^{+}$, $\tilde{u}_i=u_i^{-1}$, $\tilde{w}_i=u_i$, if $(\beta_i^{-})_{\sigma}=0$.
\end{itemize}
Since we have a pseudo-resolution the vectors $\tilde{\alpha}_1,\ldots,\tilde{\alpha}_r$ are totally ordered with respect to $\preceq$,
so that there exists a permutation $\varepsilon$ such that
\begin{equation*}
\tilde{\alpha}_{\varepsilon(1)}\preceq\tilde{\alpha}_{\varepsilon(2)}\preceq
\cdots \tilde{\alpha}_{\varepsilon(r)}
\end{equation*}
If all $f_i$ are binomials then the total trasform of $\mathfrak{a}$ in $U_{\sigma}$ is the ideal
\begin{equation} \label{EqCaso1}
\left\langle
x^{\tilde{\alpha}_{\varepsilon(1)}}(1-\tilde{u}_{\varepsilon(1)} x^{\tilde{\beta}_{\varepsilon(1)}}),
\ldots
x^{\tilde{\alpha}_{\varepsilon(r)}}(1-\tilde{u}_{\varepsilon(r)} x^{\tilde{\beta}_{\varepsilon(r)}})
\right\rangle
\end{equation}
If some $f_i$ is a monomial,
set $r'$ to be the minimum index such that $f_{\varepsilon(r')}$ is a monomial
and the total trasform of $\mathfrak{a}$ in $U_{\sigma}$ is the ideal
\begin{equation} \label{EqCaso2}
\left\langle
x^{\tilde{\alpha}_{\varepsilon(1)}}(1-\tilde{u}_{\varepsilon(1)} x^{\tilde{\beta}_{\varepsilon(1)}}),
\ldots
x^{\tilde{\alpha}_{\varepsilon(r'-1)}}(1-\tilde{u}_{\varepsilon(r'-1)} x^{\tilde{\beta}_{\varepsilon(r'-1)}}),
x^{\tilde{\alpha}_{\varepsilon(r')}}
\right\rangle
\end{equation}
Pseudo resolution will reduce the computation of the log-canonical threshold of a binomial ideal to the cases 
(\ref{EqCaso1}) and (\ref{EqCaso2}), see Proposition~\ref{PropLctAfin}.
\end{Remark}

The proof of Proposition~\ref{PseudoExist} is based on the existence of log-resolution of monomial ideals.
Now we define some invariants that will produce a constructive pseudo-resolution where every step is a combinatorial blowing up with a two codimensional center.

\begin{Definition} \label{DefLf}
Fix a regular fan $\Sigma$ in $N$ and $\Xi=\{v_1,\ldots,v_m\}$ a set of vertices of $\Sigma$.

We define the function $L=L_{\Sigma}:\mathbb{Z}^m\to\mathbb{Z}_{\geq 0}$, for $\beta=(\beta_1,\ldots,\beta_m)\in\mathbb{Z}^m$ set
\begin{equation*}
    L(\beta)=\max\{|\beta_{i}-\beta_{j}|,\ \beta_{i}\beta_{j}<0,
    \text{ and } 
    \{v_{i},v_{j}\} \text{ generate a 2-dimensional cone in } \Sigma\}.
\end{equation*}
If $\sigma\in\Sigma$ is a cone then we set
\begin{equation*}
    L(\beta,\sigma)=\max\{|\beta_{i}-\beta_{j}|,\ \beta_{i}\beta_{j}<0,
    \text{ and } 
    \{v_{i},v_{j}\} \text{ generate a 2-dimensional face of } \sigma\}.
\end{equation*}
If $f=v^{\beta^{+}}-uv^{\beta^{-}}$ is an irreducible binomial in $W_{\Sigma}$ then set 
$L(f)=L(\beta)$.

In general, if $v^a$ and $v^b$ are two monomials in $W_{\Sigma}$ then we set
$L(\{v^a,v^b\})=L(a-b)$.
\end{Definition}

Note that $L(\beta)=\max\{L(\beta,\sigma)\mid \sigma\in\Sigma\}$ and
\begin{equation*}
    L(\beta)=\max\{|\beta_{i}|+|\beta_{j}|,\ \beta_{i}\beta_{j}<0,
    \text{ and } 
    \{v_{i},v_{j}\} \text{ generate a 2-dimensional cone in } \Sigma\}.
\end{equation*}

\begin{Example} \label{ExLf}
Consider $W=\Spec(k[x_1,x_2,x_3,x_4,x_5])$, let $f=x_1^3 x_2 x_3-x_4^3 x_5^5$ be a binomial. The fan $\Sigma$ has vertices $v_1,v_2,v_3,v_4,v_5$ and
the irreducible binomial $f$ is represented by $\beta=(3,1,1,-3,-5)$.

We have that $L(\beta)=8$, and this value is given by the pair $\{v_1,v_5\}$.
\end{Example}

Function $L(f)$ was introduced in \cite{Zeillinger2006} in order to give a solution to Hironaka's polyhedra game. See also \cite{Goward2005} where the same invariant appears.
The two-codimensional centers that will be allowed in our procedure will be those pairs appearing in the definition of the function $L_{\Sigma}$.
\begin{Definition}
A 2-dimensional face $\{v_{i},v_{j}\}$ of $\Sigma$ is \emph{permissible} for 
an irreducible binomial $f=v^{\beta^{+}}-uv^{\beta^{-}}$ iff $\beta_{i}\beta_{j}<0$.

In general if $v^a$ and $v^b$ are two monomials then we say that $\{v_{i},v_{j}\}$ is \emph{permissible} for the pair of monomials $\{v^a,v^b\}$ if $(a_i-b_i)(a_j-b_j)<0$.
\end{Definition}

\begin{Proposition} \label{PropLnocrece}
Let $f$ be an irreducible binomial in $W_{\Sigma}$ and let $\{v_i,v_j\}$ be a permissible 2-dimensional face of $\Sigma$ for $f$.
If $\Sigma'$ is the star subdivision of $\Sigma$ with center $\{v_i,v_j\}$ and $f'$ is
the strict transform of $f$, then $L_{\Sigma'}(f')\leq L_{\Sigma}(f)$.
\end{Proposition}

\begin{proof}
Note that $\Xi'=\Xi\cup\{v_i+v_j\}$.
If $\sigma\in\Sigma$ is a cone such that $\{v_i,v_j\}\not\subset\sigma(1)$ then $\sigma\in\Sigma'$ and $L(f',\sigma)=L(f,\sigma)$.

If $\{v_i,v_j\}\subset\sigma(1)$ then $\sigma$ is replaced in $\Sigma'$ by two cones, say $\sigma_1$ and $\sigma_2$, where $\{v_i,v_i+v_j\}\subset\sigma_1(1)$ and $\{v_j,v_i+v_j\}\subset\sigma_2(1)$.
Since $\beta_i\beta_j<0$ then $|\beta_i+\beta_j|<\max\{|\beta_i|,|\beta_j|\}$ and it follows that
\begin{equation*}
\max\{L(f',\sigma_1),L(f',\sigma_2)\}\leq L(f,\sigma).
\end{equation*}
\end{proof}

\begin{Example}
In the Example~\ref{ExLf}, the pair $\{v_1,v_5\}$ is permissible.
Let $\Sigma'$ be the star subdivision with center $\{v_1,v_5\}$.
The fan $\Sigma'$ has vertices $v_1,v_2,v_3,v_4,v_5,v_6$ and two maximal cones $\sigma'_1$ and $\sigma'_2$ generated, respectively by
$\{v_1,v_2,v_3,v_4,v_6\}$ and $\{v_2,v_3,v_4,v_5,v_6\}$. The corresponding affine open sets are
\begin{equation*}
U_{\sigma_1}=\Spec\left(k[y_1,y_2,y_3,y_4,y_6]\right),\qquad
U_{\sigma_2}=\Spec\left(k[z_2,z_3,z_4,z_5,z_6]\right),
\end{equation*}
where $x_1=y_1y_6=z_6$, $x_2=y_2=z_2$, $x_3=y_3=z_3$, $x_4=y_4=z_4$, $x_5=y_6=z_5 z_6$.

Following Proposition \ref{PropTTransMon}, set $\beta'=(3,1,1,-3,-5,-2)$ and
the strict transform of $f$ is $f'=v_1^3v_2v_3-v_4^3 v_5^5 v_6^2$.
At every chart the binomial $f'$ is
\begin{equation*}
f'_{\sigma_1}=y_1^3y_2y_3-y_4^3 y_6^2, \qquad
f'_{\sigma_2}=z_2z_3-z_4^3 z_5^5 z_6^2.
\end{equation*}
Now $L(\beta')=6<L(\beta)=8$.
The value $L(\beta')=6$ is given by the pairs $\{v_1,v_4\}$, $\{v_2,v_5\}$ and $\{v_3,v_5\}$.

Let $\Sigma''$ be the star subdivision with center $\{v_2,v_5\}$.
The fan $\Sigma''$ has vertices $v_1$, $v_2$, $v_3$, $v_4$, $v_5$, $v_6$, $v_7$ and three maximal cones.
The strict transform of $f'$ is $f''=v_1^3v_2v_3-v_4^3 v_5^5 v_6^2 v_7^4$ and
$\beta''=(3,1,1,-3,-5,-2,-4)$.

At this step $L(\beta'')=6=L(\beta')$. But note that the value $6$ is given by pairs $\{v_1,v_4\}$ and $\{v_3,v_5\}$.
\end{Example}

This last observation is a general fact, the number of pairs giving a fixed value of the function $L$, decreases after a star subdivision with center in a permissible two-dimensional face.

\begin{Definition}
If $f$ is an irreducible binomial we define the pair of positive integers $(L(f),Lp(f))$, where $Lp(f)$ is the number of pairs $i<j$ such that $L(f)=|\beta_i-\beta_j|$.
\end{Definition}

\begin{Proposition} \label{PropLldecrece}
In the situation of Proposition~\ref{PropLnocrece}, if the pair $\{v_i,v_j\}$ is such that $L(f)=|\beta_i-\beta_j|$ then
\begin{equation*}
(L(f'),Lp(f'))<(L(f),Lp(f))
\end{equation*}
for the lexicographic ordering.
\end{Proposition}

\begin{proof}
It follows from the proof of Proposition \ref{PropLnocrece}.
\end{proof}

\begin{Theorem}
Given an irreducible binomial $f=v^{\beta^{+}}-uv^{\beta^{-}}$, the log-resolution of the ideal generated by $f$ may be obtained by $2$-codimensional blowups.
\end{Theorem}

\begin{proof}
It follows from Proposition \ref{PropLldecrece}.
\end{proof}

\begin{Proposition}
If $\{v^a,v^b\}$ is a monomial ideal generated by two monomials, then the value $(L(a-b),Lp(a-b))$ may also be used to obtain a log-resolution of the ideal.
\end{Proposition}

\begin{Corollary} \label{CorPseudo}
For a m-binomial ideal, we may obtain a pseudo resolution given by a sequence of two-codimensional blowing ups constructed with function $(L,Lp)$.
\end{Corollary}

\section{Formula of lct within a pseudo-resolution}

\begin{Definition} \label{DefNewton}
Let $\alpha_1,\ldots,\alpha_m$ points in $\mathbb{R}_{\geq 0}^n$.
The Newton polyhedron of $\alpha_1,\ldots,\alpha_m$ is the convex hull in $\mathbb{R}_{\geq 0}^n$ of the set
\begin{equation*}
\bigcup_{i=1}^{m}\left(\alpha_i+\mathbb{R}_{\geq 0}^n\right).
\end{equation*}
\end{Definition}

\begin{Remark} \label{RemNewton}
Note that a point $\beta$ is in the Newton polyhedron of $\alpha_1,\ldots,\alpha_m$ if and only if there exist real numbers $t_1,\ldots,t_m$ such that
\begin{itemize}
\item $0\leq t_i\leq 1$, for $i=1,\ldots,m$,
\item $t_1+\cdots+t_m=1$ and
\item $\beta\in \left(\sum_{i=1}^{m}t_i\alpha_i\right)+\mathbb{R}_{\geq 0}^n$
\end{itemize}

The Newton polyhedron may be also expressed in terms of hyperplanes inequalities as follows:
A point $\beta\in\mathbb{R}^n_{\geq 0}$ is in the Newton polyhedron of $\alpha_1,\ldots,\alpha_m$ if and only if
\begin{equation*}
\beta\cdot v\geq \min\{\alpha_1\cdot v,\ldots,\alpha_m\cdot v\} \qquad \forall\ v\in\mathbb{R}^n_{\geq 0}.
\end{equation*}
\end{Remark}

\begin{Lemma} \cite{Howald2001} 
\label{LemHowald}
Let $U$ be a smooth affine variety of dimension $n$.
Let $z_{1},\ldots,z_{n}$ be global sections in $\OO_{U}$ such that
they are uniformizing parameters of $U$, which means that
$\dd{z_{1}},\ldots,\dd{z_{n}}$ is a basis of $\Omega_{U/k}$.

Let $\mathfrak{a}\subset\OO_{U}$ be a monomial ideal in $U$ w.r.t 
$z$'s. The ideal $\mathfrak{a}$ is generated by monomials
\begin{equation*}
    z^{\gamma_{i}}, i=1,\ldots,r.
\end{equation*}
Let $\Delta$ be a divisor defined by a monomial $z^{c}$, $\langle z^c\rangle=\OO_U(-\Delta)$.
We assume that if $z_j$ is a unit at every point in $U$ then the previous monomials do not depend on $z_j$.

Then the multiplier ideal of $\mathfrak{a}$ is
\begin{equation*}
    \mathcal{J}(U,\Delta,\mathfrak{a}^{t})=
    \langle z^{\lambda}\mid \lambda+c+\mathbf{1}\in
    \Interior(tP)\rangle
\end{equation*}
where $P$ is the Newton polyhedron (\ref{DefNewton}) of $\gamma_{i}$'s in $\mathbb{N}^n$ and $\mathbf{1}=(1,\ldots,1)$.

In particular the $\lct$ is $\dfrac{1}{m}$ where $m$ is the minimum 
number such that $m(c+\mathbf{1})\in P$.
\end{Lemma}

\begin{proof}
It is an easy generalization of Howald's formula \cite{Howald2001}, see 
\cite{Blickle2004} for more precise statement.
\end{proof}

As a direct consequence of Remark \ref{RemNewton} and Lemma \ref{LemHowald} we have the following result.
\begin{Lemma} \label{LemHowald2}
In the situation of Lemma~\ref{LemHowald}, the $\lct$ of the monomial ideal $\mathfrak{a}$ is
\begin{equation*}
\lct(U,\Delta,\mathfrak{a})=\min\left\{
\frac{(c+\mathbf{1})\cdot v}{\min\{\gamma_1\cdot v,\ldots,\gamma_r\cdot v\}}
\mid v\in\mathbb{R}_{\geq 0}^{n}\right\}.
\end{equation*}
\end{Lemma}
The minimum in Lemma~\ref{LemHowald2} may be computed using linear programming, see \cite[Prop.~1.3]{ShibutaTakagi2009}.

Hyperbolic equations will play a key role when computing the log-canonical-threshold of binomial ideals.
These ideals have been characterized in \cite[Th. 2.1]{EisenbudSturmfels1996}.
We want to describe the behavior of such ideals when adding a new generator.

\begin{Lemma} \label{LemHiperb} 
Let $\beta_1,\ldots,\beta_r\subset\mathbb{Z}^n_{\geq 0}$.
Consider the hyperbolic binomials
\begin{equation*}
1-u_1 x^{\beta_1},\ldots,1-u_r x^{\beta_r}\in k[x_1,\ldots,x_n],
\end{equation*}
where $u_i\in k^{\ast}$, $i=1,\ldots,r$.

Let $\mathfrak{a}_r$ be the ideal generated by $1-u_1 x^{\beta_1},\ldots,1-u_r x^{\beta_r}$ in $k[x_1,\ldots,x_n]$.

Set $y_i=1-u_i x^{\beta_i}$, for $i=1,\ldots,r$, and let $U=\Spec(k[x_1,\ldots,x_n]_h)$  where $h=x^{\beta_1+\cdots+\beta_r}$.
Assume that $\beta_1,\ldots,\beta_r$ are $\mathbb{Q}$-linearly independent and set
\begin{equation*} \label{EqHomBeta}
\varphi:\left\langle \beta_1,\ldots,\beta_r \right\rangle_{\mathbb{Z}}\to k^{\ast}
\end{equation*}
to be the homomorphism from the lattice generated by $\beta_1,\ldots,\beta_r$ to the multiplicative group $k^{\ast}$ defined by $\varphi(\beta_i)=u_i$, $i=1,\ldots,r$.
\begin{enumerate}
\item $y_1,\ldots,y_r$ is part of a set of uniformizing parameters in $U$,

\item Let $\gamma\in\mathbb{Z}_{\geq 0}^n$ be a vector with non negative entries and $w\in k^{\ast}$.
The hyperbolic binomial $1-wx^{\gamma}$ is in the ideal $\mathfrak{a}_r$ if and only if
\begin{equation*}
\gamma\in\left\langle \beta_1,\ldots,\beta_r\right\rangle_{\mathbb{Z}}
\qquad\text{and}\qquad
\varphi(\gamma)=w.
\end{equation*}
\end{enumerate}
\end{Lemma}

\begin{proof} 
(1) and (2) are proved in \cite[Th. 2.1]{EisenbudSturmfels1996}, nevertheless for (1) one can check it directly. The log-jacobian matrix of $y_1,\ldots,y_r$ is
\begin{equation*}
\left(x_j\frac{\partial y_i}{\partial x_j}\right)_{i,j}=
\left(
\begin{array}{ccc}
\beta_{1,1}u_1 x^{\beta_1} & \cdots & \beta_{1,n}u_1 x^{\beta_1} \\ 
\vdots &  & \vdots \\ 
\beta_{r,1}u_r x^{\beta_r} & \cdots & \beta_{r,n}u_r x^{\beta_r}
\end{array}
\right),
\end{equation*}
and this matrix has rank $r$ in $k[x_1,\ldots,x_n]_h$ if and only if $\beta_1,\ldots,\beta_r$ are linearly independent.
\end{proof}
Note that in the proof of Lemma \ref{LemHiperb}, only (1) needs that the field $k$ has to be of characteristic zero.
The result is also true for $k$ of characteristic $p>0$, if we assume that $\bar{\beta}_1,\ldots,\bar{\beta}_r$ are $\mathbb{F}_p$-linear independent.
\medskip

We want to study the description of the ideal $\mathfrak{a}_r$ in Lemma~\ref{LemHiperb} when adding a
new hyperbolic binomial $1-u_{r+1}x^{\beta_{r+1}}$.
In general we will not have a global description, but we are able to find open sets of $U$ with explicit description of the resulting ideal.
The following lemma generalizes this setting by adding monomials $x^{\alpha_i}$ to the generators.

\begin{Lemma} \label{LemHiper2}
Let $\beta_1,\ldots,\beta_r\in\mathbb{Z}_{\geq 0}^n$ be $\mathbb{Q}$-linearly independent as in Lemma \ref{LemHiperb}.
Let $\beta_{r+1}\in\mathbb{Z}_{\geq 0}^n$ be a vector, $u_1,\ldots,u_r,u_{r+1}\in k^{\ast}$ be units in $k$
and $\alpha_1,\ldots,\alpha_r,\alpha_{r+1}\in\mathbb{Z}_{\geq 0}^n$ be vectors such that
$\beta_{r+1}$ is $\mathbb{Q}$-linearly dependent on $\beta_1,\ldots,\beta_r$ and
\begin{equation*}
\alpha_1\preceq\alpha_2\preceq\cdots \preceq\alpha_r\preceq\alpha_{r+1}
\end{equation*}
Set $h=x^{\beta_1+\cdots+\beta_r}$ and consider the sheaves of ideals in $U=\Spec(k[x_1,\ldots,x_n]_h)$ defined globally in $U$ by:
\begin{gather*}
\mathfrak{a}_r=\left\langle x^{\alpha_1}(1-u_1x^{\beta_1}),\ldots,x^{\alpha_r}(1-u_r x^{\beta_r}) \right\rangle \\
\mathfrak{a}_{r+1}=
\left\langle x^{\alpha_1}(1-u_1x^{\beta_1}),\ldots,x^{\alpha_r}(1-u_r x^{\beta_r}),
x^{\alpha_{r+1}}(1-u_{r+1}x^{\beta_{r+1}}) \right\rangle.
\end{gather*}
There are two open sets $U_1$ and $U_2$ in $U$ such that
\begin{enumerate}
\item $U_1\cup U_2=U$,

\item $\mathfrak{a}_{r+1}\OO_{U_1}=\mathfrak{a}_r\OO_{U_1}$ and

\item $\mathfrak{a}_{r+1}\OO_{U_2}=
\left(\mathfrak{a}_r+\langle x^{\alpha_{r+1}} \rangle\right)\OO_{U_2}$.
\end{enumerate}
\end{Lemma}

\begin{proof}
Let $\varphi:\langle\beta_1,\ldots,\beta_r\rangle_{\mathbb{Z}}\to k^{\ast}$ be the homomorphism defined as in Lemma \ref{EqHomBeta}.
By assumption,
there exist a positive integer $q$ and integers $\lambda_1,\ldots,\lambda_r$ such that
\begin{equation*}
q\beta_{r+1}=\lambda_1\beta_1+\cdots+\lambda_r\beta_r.
\end{equation*}
If $u_{r+1}^q\neq \varphi(q\beta_{r+1})$ then by Lemma \ref{LemHiperb} we have
\begin{equation*}
x^{\alpha_{r+1}}(1-u_{r+1}^q x^{q\beta_{r+1}})\in \mathfrak{a}_{r+1} 
\qquad\text{and}\qquad 
x^{\alpha_{r+1}}(1-\varphi(q\beta_{r+1})x^{q\beta_{r+1}})\in\mathfrak{a}_r
\end{equation*}
and we have that
$x^{\alpha_{r+1}}\in\mathfrak{a}_{r+1}\OO_U$. In this case we may set $U_2=U$, $U_1=\emptyset$.

If $u_{r+1}^q=\varphi(q\beta_{r+1})$, by Lemma \ref{LemHiperb} we have that 
$x^{\alpha_{r+1}}(1-u_{r+1}^{q}x^{q\beta_{r+1}})\in \mathfrak{a}_r$. Set
\begin{equation*}
g=\frac{1-u_{r+1}^{q}x^{q\beta_{r+1}}}{1-u_{r+1}x^{\beta_{r+1}}}\in k[x_1,\ldots,x_n]
\qquad\text{and note that}\qquad
\langle g,1-u_{r+1}x^{\beta_{r+1}}\rangle=1.
\end{equation*}
Set $U_1$ and $U_2$ the open sets of $U$ given by localization at $g$ and $1-u_{r+1}x^{\beta_{r+1}}$ respectively and we obtain the result.
\end{proof}
Note that if the number $q$ in the proof is $q=1$ then $U_1=U$.
\medskip

The ideals obtained after performing a pseudo-resolution $W'\to W$ will be, in each affine open set, either as in (\ref{EqCaso1}) or as in (\ref{EqCaso2}).

Using Lemma \ref{LemHiper2} we are able to refine the usual affine open covering of $W'$ such that at every open set the ideal is monomial with respect to a set of uniformizing parameters.

\begin{Definition} \label{DefBetaUComp}
Let $\beta_1,\ldots,\beta_r\in\mathbb{Z}^n$ be vectors and $u_1,\ldots,u_r\in k^{\ast}$ be units.

We say that $\beta_1,\ldots,\beta_r$ and $u_1,\ldots,u_r$ are \emph{compatible} if there is an homomorphism
\begin{equation*}
\varphi:\langle \beta_1,\ldots,\beta_r\rangle_{\mathbb{Z}}\to k^{\ast}
\end{equation*}
such that $\varphi(\beta_i)=u_i$ for all $i=1,\ldots,r$.
\end{Definition}
Note that the homomorphism $\varphi$, if it exists, it is unique. 
Note also that we do not require $\beta_1,\ldots,\beta_r$ to be linearly independent.

\begin{Remark}
Assume that $\beta_1,\ldots,\beta_{r-1}$ and $u_1,\ldots,u_{r-1}$ are compatible as above and let $\varphi$ the corresponding homomorphism,
$\varphi:\langle \beta_1,\ldots,\beta_{r-1}\rangle_{\mathbb{Z}}\to k^{\ast}$.

If $\beta_r$ is $\mathbb{Q}$-linearly independent of $\beta_1,\ldots,\beta_{r-1}$ then
$\beta_1,\ldots,\beta_r$ and $u_1,\ldots,u_r$ are compatible.

Assume $\beta_r$ is $\mathbb{Q}$-linearly dependent of $\beta_1,\ldots,\beta_{r-1}$, as in the proof of Lemma \ref{LemHiper2}, there are integers $q>0$ and $\lambda_1,\ldots,\lambda_{r-1}$ such that
\begin{equation*}
q\beta_r=\lambda_1\beta_1+\cdots+\lambda_{r-1}\beta_{r-1}
\end{equation*}
In this case $\beta_1,\ldots,\beta_r$ and $u_1,\ldots,u_r$ are compatible if and only if $\varphi(q\beta_r)=u_r^q$.
\end{Remark}

\begin{Nada}
We recall that after a pseudo resolution the total transform of our ideal will be as in (\ref{EqCaso1}) or as in (\ref{EqCaso2}). This means that we will obtain an ordered sequence of exponents $\beta_1,\ldots,\beta_r$ together with a sequence in $k^{\ast}$, say $u_1,\ldots,u_r$.
In the proof Lemma~\ref{LemGenLocMonom} we will use Lemma~\ref{LemHiper2} in order to see that the corresponding ideal is locally monomial at every open set of an open covering constructed in terms of $\beta$'s and $u$'s.
So that we need to check the condition to be compatible in Definition \ref{DefBetaUComp} of
$\beta_1,\ldots,\beta_i$ and $u_1,\ldots,u_i$ for every $i=1,\ldots,r$.
In the next remark we will define $\mathcal{N}(\beta,u)$ which will allow to construct explicitly the open covering.
\end{Nada}

\begin{Remark} \label{RemDefni}

Let $\beta_1,\ldots,\beta_r\in\mathbb{Z}^n$ be vectors and $u_1,\ldots,u_r\in k^{\ast}$ be units.
We associate to the ordered sequences $(\beta_1,\ldots,\beta_r)$ and $(u_1,\ldots,u_r)$ a sequence of integers $\mathcal{N}(\beta,u)$
\begin{equation*}
\mathcal{N}(\beta,u)=\{1=n_1<n_2<\cdots< n_s\leq \bar{r}\leq r\}
\end{equation*}
which is characterized by the following properties:
\begin{itemize}
\item $\beta_{n_1},\ldots,\beta_{n_s}$ are $\mathbb{Q}$-linearly independent.

\item If $j$ is such that $n_i<j<n_{i+1}$ for some $i<s$, then
$\beta_j$ is $\mathbb{Q}$-linearly independent of $\beta_{n_1},\ldots,\beta_{n_i}$.

\item If $j$ is such that $n_s<j\leq\bar{r}$ then
$\beta_j$ is $\mathbb{Q}$-linearly independent of $\beta_{n_1},\ldots,\beta_{n_s}$.

\item $\beta_1,\beta_2,\ldots,\beta_{\bar{r}}$ and $u_1,u_2,\ldots,u_{\bar{r}}$ are compatible
(\ref{DefBetaUComp}) and $\bar{r}$ is the maximum index with this property.
\end{itemize}

The sequence $n_1<\cdots<n_s\leq\bar{r}$ may be constructed inductively.
Assume that for some $\ell<r$ we have contructed the sequence
\begin{equation*}
\mathcal{N}(\beta_1,\ldots,\beta_{\ell};u_1,\ldots,u_{\ell})=
\{1=n_1<n_2<\cdots<n_{s'}\leq\bar{r}'\leq \ell\}
\end{equation*}
If $\bar{r}'<\ell$ then we have done and
\begin{equation*}
\mathcal{N}(\beta_1,\ldots,\beta_{r};u_1,\ldots,u_{r})=
\mathcal{N}(\beta_1,\ldots,\beta_{\ell};u_1,\ldots,u_{\ell})
\end{equation*}
Assume that $\bar{r}'=\ell$.
If $\beta_{\ell+1}$ is $\mathbb{Q}$-linearly independent of
$\beta_{n_1},\ldots,\beta_{n_{s'}}$ then set $n_{s''}=\ell+1$, $\bar{r}''=\ell+1$ and
\begin{equation*}
\mathcal{N}(\beta_1,\ldots,\beta_{\ell},\beta_{\ell+1};u_1,\ldots,u_{\ell},u_{\ell+1})=
\{1=n_1<n_2<\cdots<n_{s'}<n_{s''}=\bar{r}''=\ell+1\}
\end{equation*}
If $\beta_{\ell+1}$ is $\mathbb{Q}$-linearly dependent of
$\beta_{n_1},\ldots,\beta_{n_{s'}}$, there are two cases:
\begin{itemize}
\item If $\beta_1,\beta_2,\ldots,\beta_{\ell},\beta_{\ell+1}$ and
$u_1,u_2,\ldots,u_{\ell},u_{\ell+1}$ are compatible then set $\bar{r}''=\ell+1$ and
\begin{equation*}
\mathcal{N}(\beta_1,\ldots,\beta_{\ell},\beta_{\ell+1};u_1,\ldots,u_{\ell},u_{\ell+1})=
\{1=n_1<n_2<\cdots<n_{s'}<\bar{r}''=\ell+1\}
\end{equation*}
\item If $\beta_1,\beta_2,\ldots,\beta_{\ell},\beta_{\ell+1}$ and
$u_1,u_2,\ldots,u_{\ell},u_{\ell+1}$ are not compatible then
\begin{equation*}
\mathcal{N}(\beta_1,\ldots,\beta_{r};u_1,\ldots,u_{r})=
\mathcal{N}(\beta_1,\ldots,\beta_{\ell};u_1,\ldots,u_{\ell})
\end{equation*}
and we have done.
\end{itemize}
\end{Remark}
Note that the sequence $n_1<\cdots<n_s\leq\bar{r}$ depends on the ordering of the $\beta's$ and $u's$.

The sequence $n_1<\cdots<n_s$ will allow to describe generators of a binomial ideal such that it will be a monomial ideal at each open set of a suitable open cover.
\medskip

The next lemma proves that every ideal as in (\ref{EqCaso1}) or as in (\ref{EqCaso2}) is locally a monomial ideal.
\begin{Lemma} \label{LemGenLocMonom}
Let $\beta_1,\ldots,\beta_r\in\mathbb{Z}_{\geq 0}^n$ be vectors and $u_1,\ldots,u_r\in k^{\ast}$ be units.

Let $\alpha_1,\ldots,\alpha_r,\alpha_{r+1}\in\mathbb{Z}_{\geq 0}^n$ be vectors such that
\begin{equation*}
\alpha_1\preceq\alpha_2\preceq\cdots\preceq \alpha_r\preceq\alpha_{r+1}
\end{equation*}
Set the ideals in $k[x_1,\ldots,x_n]$:
\begin{gather*}
\mathfrak{a}=\left\langle
x^{\alpha_1}(1-u_1x^{\beta_1}),\ldots, x^{\alpha_r}(1-u_rx^{\beta_r})
\right\rangle \\
\mathfrak{b}=\left\langle
x^{\alpha_1}(1-u_1x^{\beta_1}),\ldots x^{\alpha_r}(1-u_rx^{\beta_r}),
x^{\alpha_{r+1}} \right\rangle
\end{gather*}
Let $\mathcal{N}(\beta,u)$ be the sequence defined in Remark \ref{RemDefni}
\begin{equation*}
\mathcal{N}(\beta,u)=\{1=n_1<n_2<\cdots< n_s\leq \bar{r}\leq r\}.
\end{equation*}
There exist an open covering of
$\mathbb{A}^n=U_1\cup U_2\cup\cdots\cup U_{\bar{r}}\cup U_{\bar{r}+1}$
such that:
\begin{enumerate}
\item For $j=1,\ldots,\bar{r}$
\begin{multline*}
\mathfrak{a}\OO_{U_j}=\mathfrak{b}\OO_{U_j}= \\
\left\langle 
x^{\alpha_{n_1}}(1-u_{n_1}x^{\beta_{n_1}}),x^{\alpha_{n_2}}(1-u_{n_2}x^{\beta_{n_2}}),
\ldots, x^{\alpha_{n_{i(j)}}}(1-u_{n_{i(j)}}x^{\beta_{n_{i(j)}}}),
x^{\alpha_j}\right\rangle,
\end{multline*}
where $i(j)$ is the maximum index with $n_{i(j)}< j$.

\item If $\bar{r}<r$ then
\begin{multline*}
\mathfrak{a}\OO_{U_{\bar{r}+1}}=\mathfrak{b}\OO_{U_{\bar{r}+1}}= \\
\left\langle 
x^{\alpha_{n_1}}(1-u_{n_1}x^{\beta_{n_1}}),x^{\alpha_{n_2}}(1-u_{n_2}x^{\beta_{n_2}}),
\ldots, x^{\alpha_{n_s}}(1-u_{n_s}x^{\beta_{n_s}}),
x^{\alpha_{\bar{r}+1}}\right\rangle.
\end{multline*}

\item If $\bar{r}=r$ then
\begin{gather*}
\mathfrak{a}\OO_{U_{\bar{r}+1}}=
\left\langle 
x^{\alpha_{n_1}}(1-u_{n_1}x^{\beta_{n_1}}),x^{\alpha_{n_2}}(1-u_{n_2}x^{\beta_{n_2}}),
\ldots, x^{\alpha_{n_s}}(1-u_{n_s}x^{\beta_{n_s}})\right\rangle, \\
\mathfrak{b}\OO_{U_{\bar{r}+1}}=
\left\langle 
x^{\alpha_{n_1}}(1-u_{n_1}x^{\beta_{n_1}}),x^{\alpha_{n_2}}(1-u_{n_2}x^{\beta_{n_2}}),
\ldots, x^{\alpha_{n_s}}(1-u_{n_s}x^{\beta_{n_s}}),
x^{\alpha_{r+1}}\right\rangle.
\end{gather*}

\item At every open set $U_j$, $j=1,\ldots,\bar{r},\bar{r}+1$ there is a set of uniformizing parameters such that the ideals $\mathfrak{a}\OO_{U_j}$ and $\mathfrak{b}\OO_{U_j}$ are monomial ideals with respect to the uniformizing parameters and moreover
\begin{gather*}
U_j\subset \{(1-u_jx^{\beta_j})\neq 0,\ x^{\beta_1+\cdots+\beta_{j-1}}\neq 0\},
\qquad j=1,\ldots,\bar{r} \\
U_{\bar{r}+1}\subset \{x^{\beta_1+\cdots+\beta_{\bar{r}}}\neq 0\}
\end{gather*}
\end{enumerate}
\end{Lemma}

\begin{proof}
Set $y_i=1-u_i x^{\beta_i}$ for $i=1,\ldots,r$.
Let $U_1$ be the open set $\{y_1\neq 0\}$ and $V_1$ be the open set $\{x^{\beta_1}\neq 0\}$
Note that $\mathbb{A}^n=U_1\cup V_1$ and
\begin{equation*}
\mathfrak{a}\OO_{U_1}=\mathfrak{b}\OO_{U_1}=
\left\langle x^{\alpha_{n_1}}\right\rangle
\end{equation*}
The nex step is to cover $V_1$ by open sets $U_2$ and $V_2$.
If $n_2>2$ then by Lemma \ref{LemHiper2} there is an open covering of $V_1$, $U_2\cup V_2$, where $U_2=\{y_2\neq 0\}\cap V_1$ and
\begin{gather*}
\mathfrak{a}\OO_{U_2}=\mathfrak{b}\OO_{U_2}=
\left\langle x^{\alpha_{n_1}}y_{n_1}, x^{\alpha_2}\right\rangle \\
\mathfrak{a}\OO_{V_2}=
\left\langle x^{\alpha_{n_1}}y_{n_1}, x^{\alpha_3}y_3,\ldots,x^{\alpha_r}y_r\right\rangle \\
\mathfrak{b}\OO_{V_2}=
\left\langle x^{\alpha_{n_1}}y_{n_1}, x^{\alpha_3}y_3,\ldots,x^{\alpha_r}y_r,
x^{\alpha_{r+1}}\right\rangle
\end{gather*}
If $n_2=2$ then set $U_2=\{y_2\neq 0,x^{\beta_1}\neq 0\}$ and $V_2=\{x^{\beta_1+\beta_2}\neq 0\}$
\begin{gather*}
\mathfrak{a}\OO_{U_2}=\mathfrak{b}\OO_{U_2}=
\left\langle x^{\alpha_{n_1}}y_1, x^{\alpha_2}\right\rangle \\
\mathfrak{a}\OO_{V_2}=
\left\langle x^{\alpha_{n_1}}y_{n_1}, x^{\alpha_{n_2}}y_{n_2}, 
x^{\alpha_3}y_3,\ldots,x^{\alpha_r}y_r\right\rangle \\
\mathfrak{b}\OO_{V_2}=
\left\langle x^{\alpha_{n_1}}y_{n_1}, x^{\alpha_{n_2}}y_{n_2},
x^{\alpha_3}y_3,\ldots,x^{\alpha_r}y_r,
x^{\alpha_{r+1}}\right\rangle
\end{gather*}
In both cases $U_2\subset\{y_2\neq 0,x^{\beta_1}\neq 0\}$ and
$V_2\subset\{x^{\beta_1+\beta_2}\neq 0\}$.
We have that $y_1$ is part of a set of uniformizing parameters in $U_2$.
\medskip

Note that for every $\ell$ either
\begin{itemize}
\item $n_{i(\ell)}<\ell\leq n_{i(\ell)+1}$ or
\item $i(\ell)=s$. In this case $n_s<\ell\leq\bar{r}$.
\end{itemize}

We proceed by induction on $\ell$.
Assume that for some $\ell<\bar{r}$ we have an open covering of $\mathbb{A}^n$, say
$U_1\cup U_2\cup \cdots\cup U_{\ell}\cup V_{\ell}$, such that
\begin{enumerate}
\item For $j=1,\ldots,\ell$
\begin{equation*}
\mathfrak{a}\OO_{U_j}=\mathfrak{b}\OO_{U_j}=
\left\langle x^{\alpha_{n_1}}y_{n_1},x^{\alpha_{n_2}}y_{n_2},\ldots,
x^{\alpha_{n_{i(j)}}}y_{n_{i(j)}}, x^{\alpha_j}\right\rangle
\end{equation*}

\item If $\ell<n_{i(\ell)+1}$ (or $i(\ell)=s$) then
\begin{eqnarray*}
\mathfrak{a}\OO_{V_{\ell}} & = & 
\left\langle x^{\alpha_{n_1}}y_{n_1},\ldots,x^{\alpha_{n_{i(\ell)}}}y_{n_{i(\ell)}},
 x^{\alpha_{\ell+1}}y_{\ell+1},\ldots,x^{\alpha_r}y_r\right\rangle \\
\mathfrak{b}\OO_{V_{\ell}} & = &
\left\langle x^{\alpha_{n_1}}y_{n_1},\ldots,x^{\alpha_{n_{i(\ell)}}}y_{n_{i(\ell)}},
 x^{\alpha_{\ell+1}}y_{\ell+1},\ldots,x^{\alpha_r}y_r,
x^{\alpha_{r+1}}\right\rangle
\end{eqnarray*}

\item If $\ell=n_{i(\ell)+1}$ then
\begin{eqnarray*}
\mathfrak{a}\OO_{V_{\ell}} & = & 
\left\langle x^{\alpha_{n_1}}y_{n_1},\ldots,x^{\alpha_{n_{i(\ell)}}}y_{n_{i(\ell)}},
x^{\alpha_{n_{i(\ell)+1}}}y_{n_{i(\ell)+1}},
 x^{\alpha_{\ell+1}}y_{\ell+1},\ldots,x^{\alpha_r}y_r\right\rangle \\
\mathfrak{b}\OO_{V_{\ell}} & = &
\left\langle x^{\alpha_{n_1}}y_{n_1},\ldots,x^{\alpha_{n_{i(\ell)}}}y_{n_{i(\ell)}},
x^{\alpha_{n_{i(\ell)+1}}}y_{n_{i(\ell)+1}},
 x^{\alpha_{\ell+1}}y_{\ell+1},\ldots,x^{\alpha_r}y_r,
x^{\alpha_{r+1}}\right\rangle
\end{eqnarray*}

\item We have $V_{\ell}\subset\{x^{\beta_1+\cdots+\beta_{\ell}}\neq 0\}$ and
 $U_j\subset \{y_j\neq 0, x^{\beta_1+\cdots+\beta_{j-1}}\neq 0\}$ for $j=1,\ldots,\ell$.
\end{enumerate}
Now we cover $V_{\ell}$ with two open sets $U_{\ell+1}$ and $V_{\ell+1}$.

If $\ell+1<n_{i(\ell+1)+1}$ (or $i(\ell+1)=s$) then $\beta_{\ell+1}$ is $\mathbb{Q}$-linearly dependent of $\beta_{n_1},\ldots,\beta_{n_{i(\ell+1)}}$ and by Lemma \ref{LemHiper2} there are two open sets $U_{\ell+1}$ and $V_{\ell+1}$ such that
\begin{eqnarray*}
\mathfrak{a}\OO_{U_{\ell+1}}=\mathfrak{b}\OO_{U_{\ell+1}} &=&
\left\langle x^{\alpha_{n_1}}y_{n_1},\ldots,
x^{\alpha_{n_{i(\ell+1)}}}y_{n_{i(\ell+1)}}, x^{\alpha_{\ell+1}}\right\rangle \\
\mathfrak{a}\OO_{V_{\ell+1}} & = & 
\left\langle x^{\alpha_{n_1}}y_{n_1},\ldots,x^{\alpha_{n_{i(\ell+1)}}}y_{n_{i(\ell+1)}},
 x^{\alpha_{\ell+2}}y_{\ell+2},\ldots,x^{\alpha_r}y_r\right\rangle \\
\mathfrak{b}\OO_{V_{\ell+1}} & = &
\left\langle x^{\alpha_{n_1}}y_{n_1},\ldots,x^{\alpha_{n_{i(\ell+1)}}}y_{n_{i(\ell+1)}},
 x^{\alpha_{\ell+2}}y_{\ell+2},\ldots,x^{\alpha_r}y_r,
x^{\alpha_{r+1}}\right\rangle
\end{eqnarray*}
Note that $U_{\ell+1}\subset\{y_{\ell+1}\neq 0\}$ and since $\beta_{\ell+1}$ is $\mathbb{Q}$-linearly dependent of $\beta_1,\ldots,\beta_{\ell}$ then
$\{x^{\beta_1+\cdots+\beta_{\ell}}\neq 0\}=\{x^{\beta_1+\cdots+\beta_{\ell}+\beta_{\ell+1}}\neq 0\}$ and
$V_{\ell+1}\subset \{x^{\beta_1+\cdots+\beta_{\ell}+\beta_{\ell+1}}\neq 0\}$.

If $\ell+1=n_{i(\ell+1)+1}$ then $\beta_{\ell+1}$ is $\mathbb{Q}$-linearly independent of $\beta_{n_1},\ldots,\beta_{n_{i(\ell+1)}}$. In this case we set
$U_{\ell+1}=\{y_{\ell+1}\neq 0\}\cap V_{\ell}$,
$V_{\ell+1}=\{x^{\beta_{\ell+1}}\neq 0\}\cap V_{\ell}$ and we have
\begin{eqnarray*}
\mathfrak{a}\OO_{U_{\ell+1}}=\mathfrak{b}\OO_{U_{\ell+1}} &=&
\left\langle x^{\alpha_{n_1}}y_{n_1},\ldots,
x^{\alpha_{n_{i(\ell+1)}}}y_{n_{i(\ell+1)}},
 x^{\alpha_{\ell+1}}\right\rangle \\
\mathfrak{a}\OO_{V_{\ell+1}} & = & 
\left\langle x^{\alpha_{n_1}}y_{n_1},\ldots,x^{\alpha_{n_{i(\ell+1)}}}y_{n_{i(\ell+1)}},
x^{\alpha_{n_{i(\ell+1)+1}}}y_{n_{i(\ell+1)+1}}, \right. \\
 & & \left. \vphantom{y_{n_{i(\ell+1)+1}}}  x^{\alpha_{\ell+2}}y_{\ell+2},\ldots,x^{\alpha_r}y_r\right\rangle \\
\mathfrak{b}\OO_{V_{\ell+1}} & = &
\left\langle x^{\alpha_{n_1}}y_{n_1},\ldots,x^{\alpha_{n_{i(\ell+1)}}}y_{n_{i(\ell+1)}},
x^{\alpha_{n_{i(\ell+1)+1}}}y_{n_{i(\ell+1)+1}},\right. \\
 & & {\left. \vphantom{y_{n_{i(\ell+1)+1}}} x^{\alpha_{\ell+2}}y_{\ell+2},\ldots,x^{\alpha_r}y_r,
x^{\alpha_{r+1}}\right\rangle}
\end{eqnarray*}

For the index $\ell=\bar{r}$ we have constructed an open covering of $\mathbb{A}^n=U_1\cup U_2\cup \cdots\cup U_{\bar{r}}\cup V_{\bar{r}}$ satisfying properties (1) to (4) above.

Set $U_{\bar{r}+1}=V_{\bar{r}}$.
If $\bar{r}=r$ then we have done.

If $\bar{r}<r$ then this means that $\beta_1,\ldots,\beta_{\bar{r}},\beta_{\bar{r}+1}$ and $u_1,\ldots,u_{\bar{r}},u_{\bar{r}+1}$ are not compatible.
By Lemma \ref{LemHiper2} we have that $x^{\alpha_{\bar{r}+1}}\in \mathfrak{a}\OO_{U_{\bar{r}+1}}$ and
$x^{\alpha_{\bar{r}+1}}\in \mathfrak{b}\OO_{U_{\bar{r}+1}}$ which implies
\begin{equation*}
\mathfrak{a}\OO_{U_{\bar{r}+1}}=\mathfrak{b}\OO_{U_{\bar{r}+1}}=
\left\langle 
x^{\alpha_{n_1}}y_{n_1},x^{\alpha_{n_2}}y_{n_2},
\ldots, x^{\alpha_{n_s}}y_{n_s},
x^{\alpha_{\bar{r}+1}}\right\rangle
\end{equation*}
as required.

Finally note that open sets $U_1,\ldots,U_{\bar{r}},U_{\bar{r}+1}$ are such that
\begin{gather*}
U_j\subset \{y_j\neq 0,\ x^{\beta_1+\cdots+\beta_{j-1}}\neq 0\},
\qquad j=1,\ldots,\bar{r} \\
U_{\bar{r}+1}\subset \{x^{\beta_1+\cdots+\beta_{\bar{r}}}\neq 0\}
\end{gather*}
For $j=1,\ldots,\bar{r}$, $y_{n_1},\ldots,y_{n_{i(j)}}$ are part of a set of uniformizing parameters of $U_j$ by Lemma~\ref{LemHiperb}.
We can complete $y_{n_1},\ldots,y_{n_{i(j)}}$ it to a set of uniformizing parameters by adding some $x_{\ell}$'s, in such a way that
$\mathfrak{a}\OO_{U_j}=\mathfrak{b}\OO_{U_j}$ are monomial ideals.

Finally for $j=\bar{r}+1$ and, again by Lemma \ref{LemHiperb}, $y_{n_1},\ldots,y_{n_s}$ are part of a set of uniformizing parameters of $U_{\bar{r}+1}$ and both 
$\mathfrak{a}\OO_{U_{\bar{r}+1}}$ and $\mathfrak{b}\OO_{U_{\bar{r}+1}}$ are monomial ideals.
\end{proof}

\begin{Remark} \label{RemXunit}
Note that for $j=1,\ldots,\bar{r},\bar{r}+1$ the open set
$U_j\subset\{x^{\beta_1+\cdots+\beta_{j-1}}\neq 0\}$ and
$\beta_1,\ldots,\beta_{j-1}\in\mathbb{Z}_{\geq 0}^n$.

The construction of the open set $U_j$ following Lemma~\ref{LemGenLocMonom} is such that a variable $x_{\ell}$ is not a unit at some point of $U_j$ if and only if
\begin{equation*}
\beta_{1,\ell}=\cdots=\beta_{j-1,\ell}=0.
\end{equation*}
\end{Remark}

\begin{Lemma} \label{LemNewton2}
Let $s$, $n$ be positive integers with $s\leq n$ and 
let $\alpha_1,\ldots,\alpha_s\in\mathbb{Z}_{\geq 0}^{n-s}$ be such that 
\begin{equation*}
\alpha_1\preceq \alpha_2\preceq \cdots \preceq\alpha_s.
\end{equation*}
Consider $\mathbb{R}^{n}$ with coordinates $(x_1,\ldots,x_{n-s},y_1,\ldots,y_s)$ and the points
\begin{equation*}
\begin{array}{ccc}
P_1=(\alpha_1,1,0,\ldots,0) & = &
(\alpha_{1,1},\ldots,\alpha_{1,n-s},1,0,\ldots,0), \\ 
P_2=(\alpha_2,0,1,\ldots,0) & = &
(\alpha_{2,1},\ldots,\alpha_{2,n-s},0,1,\ldots,0), \\ 
& \vdots \\ 
P_s=(\alpha_s,0,0,\ldots,1) & = &
(\alpha_{s,1},\ldots,\alpha_{s,n-s},0,0,\ldots,1).
\end{array} 
\end{equation*}
Then the Newton polyhedron of $P_1,\ldots,P_s$, say $\mathcal{P}_s$, is defined in $\mathbb{R}_{\geq 0}^{n}$ by the inequalities
\begin{equation} \label{IneqNewton2}
\begin{gathered}
\left.
\begin{aligned}
x_i & \geq \alpha_{1,i} \\
x_i+(\alpha_{2,i}-\alpha_{1,i})y_1 & \geq \alpha_{2,i} \\
x_i+(\alpha_{3,i}-\alpha_{1,i})y_1+(\alpha_{3,i}-\alpha_{2,i})y_2 & \geq \alpha_{3,i} \\
\cdots\cdots \\
x_i+(\alpha_{j,i}-\alpha_{1,i})y_1+\cdots+(\alpha_{j,i}-\alpha_{j-1,i})y_{j-1} & \geq \alpha_{j,i} \\
\cdots\cdots \\
x_i+(\alpha_{s,i}-\alpha_{1,i})y_1+\cdots+(\alpha_{s,i}-\alpha_{s-1,i})y_{s-1} & \geq \alpha_{s,i} \\
\end{aligned}
\right\rbrace
\qquad i=1,\ldots,n-s,\\
\text{and} \qquad
y_1+y_2+\cdots+y_s \geq 1
\end{gathered}
\end{equation}
\end{Lemma}

\begin{proof}
We want to proof that a point $(x_1,\ldots,x_{n-s},y_1,\ldots,y_s)\in\mathbb{R}_{\geq 0}^{n}$ is in $\mathcal{P}_s$ if and only if $(x,y)$ satisfies inequalities (\ref{IneqNewton2}).

By Remark~\ref{RemNewton} it is equivalent to show that $(x,y)\in\mathbb{R}_{\geq 0}^{n}$ satisfies inequalities (\ref{IneqNewton2}) if and only if  there are real numbers $t_1,\ldots,t_s$, with $0\leq t_i\leq 1$ for $i=1,\ldots,s$, $t_1+t_2+\cdots+t_s=1$ and such that
\begin{equation} \label{IneqMatrNewton2}
\left(\begin{array}{cccc}
\alpha_{1,1} & \alpha_{2,1} & \cdots & \alpha_{s,1} \\
\vdots & \vdots &  & \vdots \\
\alpha_{1,n-s} & \alpha_{2,n-s} & \cdots & \alpha_{s,n-s} \\
1 & 0 & \cdots & 0 \\
0 & 1 & \cdots & 0 \\
\vdots & \vdots & \ddots & \vdots \\
0 & 0 & \cdots & 1
\end{array}\right)
\left(\begin{array}{c}
t_1 \\ t_2 \\ \vdots \\ t_s
\end{array}\right)
\preceq
\left(\begin{array}{c}
x_1 \\ \vdots \\ x_{n-s} \\ y_1 \\ \vdots \\ y_s
\end{array}\right)
\end{equation}
Assume that there exist $t_1,\ldots,t_s$ satisfying (\ref{IneqMatrNewton2}).
For $i=1,\ldots,n-s$ we have
\begin{equation*}
x_i\geq \alpha_{1,i} t_1+\alpha_{2,i} t_2+\cdots+\alpha_{s,i}t_s\geq
\alpha_{1,i}(t_1+\cdots+t_s)=\alpha_{1,i}
\end{equation*}
and we obtain the first inequality in (\ref{IneqNewton2}).

Note that from (\ref{IneqMatrNewton2}) we have $y_i\geq t_i$ for $i=1,\ldots,s$.
If $1<j\leq s$ we have 
\begin{multline*}
x_i\geq \alpha_{1,i} t_1+\alpha_{2,i} t_2+\cdots+\alpha_{s,i}t_s\geq 
\alpha_{1,i}t_1+\cdots\alpha_{j-1,i}t_{j-1}+\alpha_{j,i}(t_j+\cdots+t_s)= \\
\alpha_{1,i}t_1+\cdots+\alpha_{j-1,i}t_{j-1}+\alpha_{j,i}(1-t_1-\cdots-t_{j-1})= \\
(\alpha_{1,i}-\alpha_{j,i})t_1+\cdots+(\alpha_{j-1,i}-\alpha_{j,i})t_{j-1}+\alpha_{j,i}
\end{multline*}
\begin{equation*}
x_i+(\alpha_{j,i}-\alpha_{1,i})t_1+\cdots+(\alpha_{j,i}-\alpha_{j-1,i})t_{j-1}\geq\alpha_{j,i}
\end{equation*}
then 
\begin{multline*}
x_i+(\alpha_{j,i}-\alpha_{1,i})y_1+\cdots+(\alpha_{j,i}-\alpha_{j-1,i})y_{j-1}\geq  \\ x_i+(\alpha_{j,i}-\alpha_{1,i})t_1+\cdots+(\alpha_{j,i}-\alpha_{j-1,i})t_{j-1}\geq\alpha_{j,i}
\end{multline*}
and it follows the $j$-th inequality in (\ref{IneqNewton2}).

The last inequality in (\ref{IneqNewton2}) follows from the fact $y_i\geq t_i$:
\begin{equation*}
y_1+\cdots+y_s\geq t_1+\cdots+t_s=1
\end{equation*}

Reciprocally, assume $(x,y)$ satisfies (\ref{IneqNewton2}).
If $y_1\geq 1$ then set $t_1=1$, $t_2=\cdots=t_s=0$ and we have (\ref{IneqMatrNewton2}) since
$x_i\geq\alpha_{1,i}$ for all $i=1,\ldots,n-s$ by hypothesis.

If $y_1<1$. Note that $y_1+\cdots+y_s\geq 1$ by hypothesis. Let $j\geq 1$ be the minimum index such that $y_1+y_2+\cdots+y_{j}\geq 1$. Note that $y_1+y_2+\cdots+y_{j-1}<1$.
Set
\begin{gather*}
t_1=y_1, \quad t_2=y_2, \quad\cdots\quad, t_{j-1}=y_{j-1},\quad t_j=1-(t_1+t_2+\cdots+t_{j-1}) \\
t_{j+1}=\cdots=t_s=0.
\end{gather*}
Note that $y_i\geq t_i$ for $i=1,\ldots,s$.
We only have to check the first $n-s$ rows in (\ref{IneqMatrNewton2}). From the $j$-th inequality of (\ref{IneqNewton2})
\begin{multline*}
x_i\geq \alpha_{j,i}
+(\alpha_{1,i}-\alpha_{j,i})y_1+\cdots+(\alpha_{j-1,i}-\alpha_{j,i})y_{j-1} = \\
\alpha_{j,i}+(\alpha_{1,i}-\alpha_{j,i})t_1+\cdots+(\alpha_{j-1,i}-\alpha_{j,i})t_{j-1}= \\
\alpha_{1,i}t_1+\cdots+\alpha_{j-1,i}t_{j-1}+\alpha_{j,i}(1-(t_1+\cdots+t_{j-1}))=
\alpha_{1,i}t_1+\cdots+\alpha_{j-1,i}t_{j-1}+\alpha_{j,i}t_j
\end{multline*}
for $i=1,\ldots,n-s$ and we obtain the required inequalities of (\ref{IneqMatrNewton2}).
\end{proof}

\begin{Lemma} \label{LemNewton3}
Let $s$, $n$ be positive integers with $s\leq n$ and 
let $\alpha_1,\ldots,\alpha_s,\alpha_{s+1}\in\mathbb{Z}_{\geq 0}^{n-s}$ be such that 
\begin{equation*}
\alpha_1\preceq \alpha_2\preceq \cdots \preceq\alpha_s \preceq\alpha_{s+1}.
\end{equation*}
Consider $\mathbb{R}^{n}$ with coordinates $(x_1,\ldots,x_{n-s},y_1,\ldots,y_s)$ and the points
\begin{equation*}
\begin{array}{rcl}
P_1=(\alpha_1,1,0,\ldots,0) & = &
(\alpha_{1,1},\ldots,\alpha_{1,n-s},1,0,\ldots,0), \\ 
P_2=(\alpha_2,0,1,\ldots,0) & = &
(\alpha_{2,1},\ldots,\alpha_{2,n-s},0,1,\ldots,0), \\ 
& \vdots \\ 
P_s=(\alpha_s,0,0,\ldots,1) & = &
(\alpha_{s,1},\ldots,\alpha_{s,n-s},0,0,\ldots,1), \\
P_{s+1}=(\alpha_{s+1},0,0,\ldots,0) & = &
(\alpha_{s+1,1},\ldots,\alpha_{s+1,n-s},0,0,\ldots,0), \\
\end{array} 
\end{equation*}
Then the Newton polyhedron of $P_1,P_2,\ldots,P_s,P_{s+1}$, say $\mathcal{P}'_s$,
is defined in $\mathbb{R}_{\geq 0}^{n}$ by the inequalities
\begin{equation} \label{IneqNewton3}
\left.
\begin{aligned}
x_i & \geq \alpha_{1,i} \\
x_i+(\alpha_{2,i}-\alpha_{1,i})y_1 & \geq \alpha_{2,i} \\
x_i+(\alpha_{3,i}-\alpha_{1,i})y_1+(\alpha_{3,i}-\alpha_{2,i})y_2 & \geq \alpha_{3,i} \\
\cdots\cdots \\
x_i+(\alpha_{j,i}-\alpha_{1,i})y_1+\cdots+(\alpha_{j,i}-\alpha_{j-1,i})y_{j-1} & \geq \alpha_{j,i} \\
\cdots\cdots \\
x_i+(\alpha_{s,i}-\alpha_{1,i})y_1+\cdots+(\alpha_{s,i}-\alpha_{s-1,i})y_{s-1} & \geq \alpha_{s,i} \\
x_i+(\alpha_{s+1,i}-\alpha_{1,i})y_1+\cdots+
(\alpha_{s+1,i}-\alpha_{s,i})y_{s} & \geq \alpha_{s+1,i}
\end{aligned}
\right\rbrace
\qquad i=1,\ldots,n-s.
\end{equation}
\end{Lemma}

\begin{proof}
As in the proof of Lemma \ref{LemNewton2}, it is enough to show that a point
$$(x_1,\ldots,x_{n-s},y_1,\ldots,y_s)\in\mathbb{R}_{\geq 0}^{n}$$
satisfies inequalities (\ref{IneqNewton3}) if and only if  there are real numbers $t_1,\ldots,t_s,t_{s+1}$, with $0\leq t_i\leq 1$ for $i=1,\ldots,s,s+1$, $t_1+t_2+\cdots+t_s+t_{s+1}=1$ and such that
\begin{equation} \label{IneqMatrNewton3}
\left(\begin{array}{ccccc}
\alpha_{1,1} & \alpha_{2,1} & \cdots & \alpha_{s,1} & \alpha_{s+1,1} \\
\vdots & \vdots &  & \vdots & \vdots \\
\alpha_{1,n-s} & \alpha_{2,n-s} & \cdots & \alpha_{s,n-s} & \alpha_{s+1,n-s} \\
1 & 0 & \cdots & 0 & 0 \\
0 & 1 & \cdots & 0 & 0 \\
\vdots & \vdots & \ddots & \vdots & \vdots \\
0 & 0 & \cdots & 1 & 0
\end{array}\right)
\left(\begin{array}{c}
t_1 \\ t_2 \\ \vdots \\ t_s \\ t_{s+1}
\end{array}\right)
\preceq
\left(\begin{array}{c}
x_1 \\ \vdots \\ x_{n-s} \\ y_1 \\ \vdots \\ y_s
\end{array}\right)
\end{equation}
For $i=1,\ldots,n-s$ we have
\begin{equation*}
x_i\geq \alpha_{1,i} t_1+\alpha_{2,i} t_2+\cdots+\alpha_{s,i}t_s+\alpha_{s+1,i}t_{s+1}\geq
\alpha_{1,i}(t_1+\cdots+t_s+t_{s+1})=\alpha_{1,i}
\end{equation*}
and we obtain the first inequality in (\ref{IneqNewton3}).

Note that from (\ref{IneqMatrNewton3}) we have $y_i\geq t_i$ for $i=1,\ldots,s$.
If $1<j\leq s+1$ we have 
\begin{multline*}
x_i\geq \alpha_{1,i} t_1+\alpha_{2,i} t_2+\cdots+\alpha_{s,i}t_s+\alpha_{s+1,i}t_{s+1}\geq \\
\alpha_{1,i}t_1+\cdots\alpha_{j-1,i}t_{j-1}+\alpha_{j,i}(t_j+\cdots+t_s+t_{s+1})= \\
\alpha_{1,i}t_1+\cdots+\alpha_{j-1,i}t_{j-1}+\alpha_{j,i}(1-t_1-\cdots-t_{j-1})= \\
(\alpha_{1,i}-\alpha_{j,i})t_1+\cdots+(\alpha_{j-1,i}-\alpha_{j,i})t_{j-1}+\alpha_{j,i}
\end{multline*}
\begin{equation*}
x_i+(\alpha_{j,i}-\alpha_{1,i})t_1+\cdots+(\alpha_{j,i}-\alpha_{j-1,i})t_{j-1}\geq\alpha_{j,i}
\end{equation*}
then 
\begin{multline*}
x_i+(\alpha_{j,i}-\alpha_{1,i})y_1+\cdots+(\alpha_{j,i}-\alpha_{j-1,i})y_{j-1}\geq  \\ x_i+(\alpha_{j,i}-\alpha_{1,i})t_1+\cdots+(\alpha_{j,i}-\alpha_{j-1,i})t_{j-1}\geq\alpha_{j,i}
\end{multline*}
and it follows the $j$-th inequality in (\ref{IneqNewton3}).

Reciprocally, assume $(x,y)$ satisfies (\ref{IneqNewton3}).
If $y_1\geq 1$ then set $t_1=1$, $t_2=\cdots=t_s=t_{s+1}=0$ and we have (\ref{IneqMatrNewton3}) since
$x_i\geq\alpha_{1,i}$ for all $i=1,\ldots,n-s$ by hypothesis.

If $y_1<1$. Let $j\geq 1$ be the maximum index such that $y_1+y_2+\cdots+y_{j-1}<1$.
Note that if $y_1+\cdots+y_s<1$ then $j=s+1$.
Now set
\begin{gather*}
t_1=y_1, \quad t_2=y_2, \quad\cdots\quad, t_{j-1}=y_{j-1},\quad t_j=1-(t_1+t_2+\cdots+t_{j-1}) \\
t_{j+1}=\cdots=t_s=t_{s+1}=0.
\end{gather*}
For $i=1,\ldots,s$ we have $y_i\geq t_i$.
We only have to check the first $n-s$ rows in (\ref{IneqMatrNewton3}). From the $j$-th inequality of (\ref{IneqNewton3})
\begin{multline*}
x_i\geq \alpha_{j,i}
+(\alpha_{1,i}-\alpha_{j,i})y_1+\cdots+(\alpha_{j-1,i}-\alpha_{j,i})y_{j-1} = \\
\alpha_{j,i}+(\alpha_{1,i}-\alpha_{j,i})t_1+\cdots+(\alpha_{j-1,i}-\alpha_{j,i})t_{j-1}= \\
\alpha_{1,i}t_1+\cdots+\alpha_{j-1,i}t_{j-1}+\alpha_{j,i}(1-(t_1+\cdots+t_{j-1}))=\\
\alpha_{1,i}t_1+\cdots+\alpha_{j-1,i}t_{j-1}+\alpha_{j,i}t_j
\end{multline*}
for $i=1,\ldots,n-s$ and we obtain all inequalities in (\ref{IneqMatrNewton2}).
\end{proof}

\begin{Definition} \label{DefPsi}
Let $c,\alpha_1,\ldots,\alpha_s\in\mathbb{Z}_{\geq 0}^m$ be vectors such that
\begin{equation*}
\alpha_1\preceq\alpha_2\preceq\cdots\preceq\alpha_s
\end{equation*}
For every $\ell=1,\ldots,m$ we set
\begin{gather*}
\Psi_{\ell}(c,\alpha_1,\ldots,\alpha_s)=
\frac{c_{\ell}+1+(\alpha_{s,\ell}-\alpha_{1,\ell})+(\alpha_{s,\ell}-\alpha_{2,\ell})+\cdots
+(\alpha_{s,\ell}-\alpha_{s-1,\ell})}{\alpha_{s,\ell}} \\
\widetilde{\Psi}_{\ell}(c,\alpha_1,\ldots,\alpha_s)=
\min\left\lbrace
\Psi_{\ell}(c,\alpha_1,\ldots,\alpha_i) \mid i=1,\ldots,s
\right\rbrace
\end{gather*}
where we set $\Psi_{\ell}(c,\alpha_1,\ldots,\alpha_s)=\infty$ if $\alpha_{s,\ell}=0$.
\end{Definition}

\begin{Lemma} \label{LemPsiDecre}
Let $c,\alpha_1,\ldots,\alpha_s,\delta,\delta'\in\mathbb{Z}_{\geq 0}^m$ be vectors such that
\begin{equation*}
\alpha_1\preceq\alpha_2\preceq\cdots\preceq\alpha_s\preceq \delta\preceq\delta'
\end{equation*}
Then for every $\ell=1,\ldots,m$ we have
\begin{equation*}
\widetilde{\Psi}_{\ell}(c,\alpha_1,\ldots,\alpha_s,\delta)\geq
\widetilde{\Psi}_{\ell}(c,\alpha_1,\ldots,\alpha_s,\delta')
\end{equation*}
\end{Lemma}

\begin{proof}
Note that
\begin{gather*}
\widetilde{\Psi}_{\ell}(c,\alpha_1,\ldots,\alpha_s,\delta)=
\min\left\lbrace
\widetilde{\Psi}_{\ell}(c,\alpha_1,\ldots,\alpha_s),
\Psi_{\ell}(c,\alpha_1,\ldots,\alpha_s,\delta)
\right\rbrace, \\
\widetilde{\Psi}_{\ell}(c,\alpha_1,\ldots,\alpha_s,\delta')=
\min\left\lbrace
\widetilde{\Psi}_{\ell}(c,\alpha_1,\ldots,\alpha_s),
\Psi_{\ell}(c,\alpha_1,\ldots,\alpha_s,\delta')
\right\rbrace,
\end{gather*}
and it is enough to prove that
\begin{multline*}
\min\left\lbrace
\widetilde{\Psi}_{\ell}(c,\alpha_1,\ldots,\alpha_s),
\Psi_{\ell}(c,\alpha_1,\ldots,\alpha_s,\delta')
\right\rbrace= \\
\min\left\lbrace
\widetilde{\Psi}_{\ell}(c,\alpha_1,\ldots,\alpha_s),
\Psi_{\ell}(c,\alpha_1,\ldots,\alpha_s,\delta),
\Psi_{\ell}(c,\alpha_1,\ldots,\alpha_s,\delta')
\right\rbrace
\end{multline*}
Note that
\begin{equation*}
\Psi_{\ell}(c,\alpha_1,\ldots,\alpha_s,\delta)=
\frac{c_{\ell}+1+(\alpha_{s,\ell}-\alpha_{1,\ell})+\cdots+(\alpha_{s,\ell}-\alpha_{s-1,\ell})+
s(\delta_{\ell}-\alpha_{s,\ell})
}{\alpha_{s,\ell}+(\delta_{\ell}-\alpha_{s,\ell})}
\end{equation*}
If $\Psi_{\ell}(c,\alpha_1,\ldots,\alpha_s)\leq s$ then
\begin{gather*}
\Psi_{\ell}(c,\alpha_1,\ldots,\alpha_s)\leq
\Psi_{\ell}(c,\alpha_1,\ldots,\alpha_s,\delta), \\
\Psi_{\ell}(c,\alpha_1,\ldots,\alpha_s)\leq
\Psi_{\ell}(c,\alpha_1,\ldots,\alpha_s,\delta'),
\end{gather*}
and
\begin{multline*}
\min\left\lbrace
\widetilde{\Psi}_{\ell}(c,\alpha_1,\ldots,\alpha_s),
\Psi_{\ell}(c,\alpha_1,\ldots,\alpha_s,\delta')
\right\rbrace= \\
\min\left\lbrace
\widetilde{\Psi}_{\ell}(c,\alpha_1,\ldots,\alpha_s),
\Psi_{\ell}(c,\alpha_1,\ldots,\alpha_s,\delta),
\Psi_{\ell}(c,\alpha_1,\ldots,\alpha_s,\delta')
\right\rbrace=\\
\widetilde{\Psi}_{\ell}(c,\alpha_1,\ldots,\alpha_s)
\end{multline*}
If $\Psi_{\ell}(c,\alpha_1,\ldots,\alpha_s)> s$ then
\begin{equation*}
\Psi_{\ell}(c,\alpha_1,\ldots,\alpha_s)>
\Psi_{\ell}(c,\alpha_1,\ldots,\alpha_s,\delta)>
\Psi_{\ell}(c,\alpha_1,\ldots,\alpha_s,\delta')
\end{equation*}
and
\begin{multline*}
\min\left\lbrace
\widetilde{\Psi}_{\ell}(c,\alpha_1,\ldots,\alpha_s),
\Psi_{\ell}(c,\alpha_1,\ldots,\alpha_s,\delta')
\right\rbrace= \\
\min\left\lbrace
\widetilde{\Psi}_{\ell}(c,\alpha_1,\ldots,\alpha_s),
\Psi_{\ell}(c,\alpha_1,\ldots,\alpha_s,\delta),
\Psi_{\ell}(c,\alpha_1,\ldots,\alpha_s,\delta')
\right\rbrace=\\
\min\left\lbrace
\widetilde{\Psi}_{\ell}(c,\alpha_1,\ldots,\alpha_{s-1}),
\Psi_{\ell}(c,\alpha_1,\ldots,\alpha_s,\delta')
\right\rbrace
\end{multline*}
\end{proof}

\begin{Lemma} \label{LemMonPsi2}
Let $\alpha_1,\ldots,\alpha_s\in\mathbb{Z}_{\geq 0}^{n-s}$ be such that 
\begin{equation*}
\alpha_1\preceq \alpha_2\preceq \cdots \preceq\alpha_s.
\end{equation*}
Let $U$ be a smooth affine variety of dimension $n$ and let $x_1,\ldots,x_{n-s},y_1,\ldots,y_s$ be a set of uniformizing parameters.
Consider the monomial ideal in $U$ 
\begin{equation*}
\mathfrak{a}=\left\langle
x^{\alpha_1}y_1,x^{\alpha_2}y_2,\ldots,x^{\alpha_s}y_s
\right\rangle
\end{equation*}
Let $\Delta$ be the divisor in $U$ defined by the monomial $x^{c}$ where $c\in\mathbb{Z}_{\geq 0}^{n-s}$.
We are asuming that all variables appearing in the previous expressions are not units in $\OO_U$,
which means that if some $\alpha_{i,\ell}\neq 0$ then $x_{\ell}$ is not a unit in $\OO_{U}(U)$.

The log-canonical threshold $\lct(U,\Delta,\mathfrak{a})$ is
\begin{equation*}
\lct(U,\Delta,\mathfrak{a})=
\min\left\lbrace s,\widetilde{\Psi}_{\ell}(c,\alpha_1,\ldots,\alpha_s) \mid
\ell=1,\ldots,n-s \right\rbrace
\end{equation*}
\end{Lemma}

\begin{proof}
Note first that the Newton polyhedron $\mathcal{P}$ of the ideal $\mathfrak{a}$ is the one of
Lemma~\ref{LemNewton2}. Set $(c,0)=(c_1,\ldots,c_{n-s},0,\ldots,0)\in\mathbb{Z}_{\geq 0}^{n}$,
by Lemma~\ref{LemHowald} the log-canonical threshold
$\lct(U,\Delta,\mathfrak{a})=\lambda$ where $\lambda$ is the maximum value such that
$\dfrac{1}{\lambda}((c,0)+\mathbf{1})\in\mathcal{P}$.
Using the inequalities of Lemma \ref{LemNewton2} for the point $\dfrac{1}{\lambda}((c,0)+\mathbf{1})$
we have
\begin{equation*}
\begin{gathered}
\left.
\begin{aligned}
\frac{c_{\ell}+1}{\lambda} & \geq \alpha_{1,\ell} \\
\frac{c_{\ell}+1}{\lambda}+(\alpha_{2,i}-\alpha_{1,i})\frac{1}{\lambda} & \geq \alpha_{2,\ell} \\
\frac{c_{\ell}+1}{\lambda}+(\alpha_{3,i}-\alpha_{1,i})\frac{1}{\lambda}+(\alpha_{3,i}-\alpha_{2,i})\frac{1}{\lambda} & \geq \alpha_{3,\ell} \\
\cdots\cdots \\
\frac{c_{\ell}+1}{\lambda}+(\alpha_{j,i}-\alpha_{1,i})\frac{1}{\lambda}+\cdots+(\alpha_{j,i}-\alpha_{j-1,i})\frac{1}{\lambda} & \geq \alpha_{j,\ell} \\
\cdots\cdots \\
\frac{c_{\ell}+1}{\lambda}+(\alpha_{s,i}-\alpha_{1,i})\frac{1}{\lambda}+\cdots+(\alpha_{s,i}-\alpha_{s-1,i})\frac{1}{\lambda} & \geq \alpha_{s,\ell} \\
\end{aligned}
\right\rbrace
\qquad \ell=1,\ldots,n-s,\\
\text{and} \qquad
\frac{s}{\lambda} \geq 1,
\end{gathered}
\end{equation*}
and we obtain the result.
\end{proof}

\begin{Lemma} \label{LemMonPsi3}
Let $\alpha_1,\ldots,\alpha_s,\alpha_{s+1}\in\mathbb{Z}_{\geq 0}^{n-s}$ be such that 
\begin{equation*}
\alpha_1\preceq \alpha_2\preceq \cdots \preceq\alpha_s\preceq\alpha_{s+1}.
\end{equation*}
Let $U$ be a smooth affine variety of dimension $n$ and let $x_1,\ldots,x_{n-s},y_1,\ldots,y_s$ be a set of uniformizing parameters.
Consider the monomial ideal in $U$ 
\begin{equation*}
\mathfrak{a}=\left\langle
x^{\alpha_1}y_1,x^{\alpha_2}y_2,\ldots,x^{\alpha_s}y_s, x^{\alpha_{s+1}}
\right\rangle
\end{equation*}
Let $\Delta$ be the divisor in $U$ defined by the monomial $x^{c}$ where $c\in\mathbb{Z}_{\geq 0}^{n-s}$.
The log-canonical threshold $\lct(U,\Delta,\mathfrak{a})$ is
\begin{equation*}
\lct(U,\Delta,\mathfrak{a})=
\min\left\lbrace \widetilde{\Psi}_{\ell}(c,\alpha_1,\ldots,\alpha_s,\alpha_{s+1}) \mid
\ell=1,\ldots,n-s \right\rbrace
\end{equation*}
\end{Lemma}

\begin{proof}
It is analogous to the proof of Lemma \ref{LemMonPsi2}, using Lemma~\ref{LemNewton3}.
\end{proof}

Recall that after a pseudo-resolution we will find an ideal as in (\ref{EqCaso1}) or (\ref{EqCaso2}), see also Lemma~\ref{LemGenLocMonom}.
Then by Proposition \ref{PropLctCover} it will be enough to compute the log-canonical threshold for the ideals in the following proposition.
\begin{Proposition} \label{PropLctAfin}
Let $U$ be a smooth toric affine variety of dimension $n$. We known
that $U\cong\Spec(k[x_{1},\ldots,x_{n}]_{h})$, where $h$ is a product 
of some variables $x_{i}$.
Let $\beta_1,\ldots,\beta_r\in\mathbb{Z}_{\geq 0}^n$ be vectors and $u_1,\ldots,u_r\in k^{\ast}$ be units.

Let $\alpha_1,\ldots,\alpha_r,\alpha_{r+1}\in\mathbb{Z}_{\geq 0}^n$ be vectors such that
\begin{equation*}
\alpha_1\preceq\alpha_2\preceq\cdots\preceq \alpha_r\preceq\alpha_{r+1}
\end{equation*}
Set the ideals in $\OO_U$:
\begin{gather*}
\mathfrak{a}=\left\langle
x^{\alpha_1}(1-u_1x^{\beta_1}),\ldots, x^{\alpha_r}(1-u_rx^{\beta_r})
\right\rangle, \\
\mathfrak{b}=\left\langle
x^{\alpha_1}(1-u_1x^{\beta_1}),\ldots x^{\alpha_r}(1-u_rx^{\beta_r}),
x^{\alpha_{r+1}} \right\rangle.
\end{gather*}
We assume that monomials $x^{\alpha_i}$ do not depend on variables $x_j$ appearing in $h$, since these $x_j$ are units in $U$.
This means that $\alpha_{i,j}=0$ for $i=1,\ldots,r+1$ if $x_j$ divides $h$.

Let $\mathcal{N}(\beta,u)$ be the sequence defined in Remark \ref{RemDefni}
\begin{equation*}
\mathcal{N}(\beta,u)=\{1=n_1<n_2<\cdots< n_s\leq \bar{r}\leq r\}.
\end{equation*}
Let $\Delta$ be the divisor defined by the monomial $x^{c}$, $c\in\mathbb{Z}_{\geq 0}^n$.
The log-canonical thresholds of ideals $\mathfrak{b}$ and $\mathfrak{a}$ can be computed as follows:
\begin{enumerate}
\item If $\bar{r}<r$ then $\lct(U,\Delta,\mathfrak{a})=\lct(U,\Delta,\mathfrak{b})$ and
\begin{multline*}
\lct(U,\Delta,\mathfrak{a})=\lct(U,\Delta,\mathfrak{b})=\\
\min\left\lbrace \vphantom{\frac{a}{b}}
\left\lbrace
\widetilde{\Psi}_{\ell}(c,\alpha_{n_1},\ldots,\alpha_{n_i}) \mid i=1,\ldots,s,\quad 
\ell \text{ such that }\beta_{j,\ell}=0\ \forall j<n_i
\right\rbrace
\cup\right. \\
\left. \vphantom{\frac{a}{b}} \cup \left\lbrace 
\widetilde{\Psi}_{\ell}(c,\alpha_{n_1},\ldots,\alpha_{n_s},\alpha_{\bar{r}+1}) \mid 
\ell \text{ such that }\beta_{j,\ell}=0\ 
\forall j\leq\bar{r}
\right\rbrace
\right\rbrace.
\end{multline*}

\item If $\bar{r}=r$ then
\begin{equation*}
\lct(U,\Delta,\mathfrak{a})=
\min\left\lbrace
s,
\widetilde{\Psi}_{\ell}(c,\alpha_{n_1},\ldots,\alpha_{n_i}) \mid i=1,\ldots,s,\quad
\ell \text{ such that }\beta_{j,\ell}=0\ \forall j<n_i
\right\rbrace,
\end{equation*}
\begin{multline*}
\lct(U,\Delta,\mathfrak{b})=\\
\min\left\lbrace \vphantom{\frac{a}{b}}
\left\lbrace
\widetilde{\Psi}_{\ell}(c,\alpha_{n_1},\ldots,\alpha_{n_i}) \mid i=1,\ldots,s,\quad
\ell \text{ such that }\beta_{j,\ell}=0\ \forall j<n_i
\right\rbrace
\cup\right. \\
\left. \vphantom{\frac{a}{b}} \cup \left\lbrace 
\widetilde{\Psi}_{\ell}(c,\alpha_{n_1},\ldots,\alpha_{n_s},\alpha_{r+1}) \mid
\ell \text{ such that } \beta_{j,\ell}=0\ 
\forall j\leq r
\right\rbrace
\right\rbrace.
\end{multline*}
\end{enumerate}
\end{Proposition}

\begin{proof}
Let $U_1\cup\cdots\cup U_{\bar{r}}\cup U_{\bar{r}+1}$ be the open covering constructed in Lemma \ref{LemGenLocMonom}. 
Note that
\begin{gather*}
\lct(U,\Delta,\mathfrak{a})=\min\left\lbrace \lct(U_j,\Delta,\mathfrak{a}) \mid j=1,\ldots,\bar{r},\bar{r}+1 \right\rbrace,
\\
\lct(U,\Delta,\mathfrak{b})=\min\left\lbrace \lct(U_j,\Delta,\mathfrak{b}) \mid j=1,\ldots,\bar{r},\bar{r}+1 \right\rbrace.
\end{gather*}

Set $y_i=(1-u_i x^{\beta_i})$ for $i=1,\ldots,r$.

Let $j=1,\ldots,\bar{r}$. By Lemma \ref{LemGenLocMonom} we have $U_j\subset\{y_j\neq 0, x^{\beta_{1}+\cdots+\beta_{j-1}}\neq 0\}$ and
\begin{equation*}
\mathfrak{a}\OO_{U_j}=\mathfrak{b}\OO_{U_j}=
\left\langle
x^{\alpha_{n_1}}y_{n_1},x^{\alpha_{n_2}}y_{n_2},\ldots,x^{\alpha_{n_{i(j)}}}y_{n_{i(j)}},x^{\alpha_j}
\right\rangle
\end{equation*}
Lemma~\ref{LemGenLocMonom} also says that $y_{n_1},y_{n_2},\ldots,y_{n_{i(j)}}$ are part of a set of uniformizing parameters, which may be completed to a set of uniformizing parameters with suitable $x_{\ell}$.
Now we use formula in Lemma~\ref{LemMonPsi3} for $U_j$, where we forget indexes $\ell$ such that $x_{\ell}$ is a unit in $\OO_{U_j}$,
by Remark~\ref{RemXunit} we have
\begin{multline*}
\lct(U_j,\Delta,\mathfrak{a})=
\lct(U_j,\Delta,\mathfrak{b})= \\
\min\left\lbrace
\widetilde{\Psi}_{\ell}(c,\alpha_{n_1},\ldots,\alpha_{n_{i(j)}},\alpha_j) \mid
\ell \text{ such that }\beta_{1,\ell}=\cdots=\beta_{j-1,\ell}=0
\right\rbrace.
\end{multline*}
For some $i=1,\ldots,s-1$, consider the union of the open sets $U_j$ with $n_i<j\leq n_{i+1}$,
we claim that
\begin{multline*}
\lct(U_{n_i+1}\cup\cdots\cup U_{n_{i+1}},\Delta,\mathfrak{a})=
\lct(U_{n_i+1}\cup\cdots\cup U_{n_{i+1}},\Delta,\mathfrak{b})= \\
\lct(U_{n_{i+1}},\Delta,\mathfrak{a})=
\lct(U_{n_{i+1}},\Delta,\mathfrak{b})= \\
\min\left\lbrace
\widetilde{\Psi}_{\ell}(c,\alpha_{n_1},\ldots,\alpha_{n_{i}},\alpha_{n_{i+1}}) \mid
\ell \text{ such that }\beta_{1,\ell}=\cdots=\beta_{n_{i+1}-1,\ell}=0
\right\rbrace.
\end{multline*}
Note that this is consistent with the fact that for $n_i<j<n_{i+1}$ some of the open sets $U_j$ may be empty, but we always have $U_{n_i}\neq\emptyset$ for $i=1,\ldots,s$.

Our claim is a direct consequence of Lemma~\ref{LemPsiDecre} and the fact that 
$\beta_{n_i+1},\ldots,\beta_{n_{i+1}-1}$ are $\mathbb{Q}$-linearly dependent of 
$\beta_{n_1},\beta_{n_2},\ldots,\beta_{n_i}$.

If follows from the claim that
\begin{multline} \label{EqU1r}
\lct(U_{1}\cup\cdots\cup U_{\bar{r}},\Delta,\mathfrak{a})=
\lct(U_{1}\cup\cdots\cup U_{\bar{r}},\Delta,\mathfrak{b})= \\
\min\left\lbrace
\widetilde{\Psi}_{\ell}(c,\alpha_{n_1},\ldots,\alpha_{n_{i}}) \mid
i=1,\ldots,s;\quad
\ell \text{ such that }\beta_{1,\ell}=\cdots=\beta_{n_{i}-1,\ell}=0
\right\rbrace.
\end{multline}

Now we consider the open set $U_{\bar{r}}$.
If $\bar{r}<r$ then by Lemma \ref{LemGenLocMonom}
\begin{equation*}
\mathfrak{a}\OO_{U_{\bar{r}+1}}=\mathfrak{b}\OO_{U_{\bar{r}+1}}=
\left\langle
x^{\alpha_{n_1}}y_{n_1},x^{\alpha_{n_2}}y_{n_2},\ldots,x^{\alpha_{n_{s}}}y_{n_{s}},
x^{\alpha_{\bar{r}+1}}
\right\rangle
\end{equation*}
and as above, it follows from formula in Lemma \ref{LemMonPsi3} and Remark \ref{RemXunit} that
\begin{multline} \label{EqUrr2}
\lct(U_{\bar{r}+1},\Delta,\mathfrak{a})=
\lct(U_{\bar{r}+1},\Delta,\mathfrak{b})= \\
\min\left\lbrace
\widetilde{\Psi}_{\ell}(c,\alpha_{n_1},\ldots,\alpha_{n_{s}},\alpha_{\bar{r}+1}) \mid
\ell \text{ such that }\beta_{1,\ell}=\cdots=\beta_{\bar{r},\ell}=0
\right\rbrace.
\end{multline}

If $\bar{r}=r$ then by Lemma \ref{LemGenLocMonom}
\begin{gather*}
\mathfrak{a}\OO_{U_{\bar{r}+1}}=
\left\langle 
x^{\alpha_{n_1}}y_{n_1},x^{\alpha_{n_2}}y_{n_2},
\ldots, x^{\alpha_{n_s}}y_{n_s}\right\rangle, \\
\mathfrak{b}\OO_{U_{\bar{r}+1}}=
\left\langle 
x^{\alpha_{n_1}}y_{n_1},x^{\alpha_{n_2}}y_{n_2},
\ldots, x^{\alpha_{n_s}}y_{n_s},
x^{\alpha_{r+1}}\right\rangle.
\end{gather*}
Using, respectively, formulas in Lemmas \ref{LemMonPsi2} and \ref{LemMonPsi3} together with Remark \ref{RemXunit} we have
\begin{multline} \label{EqUrr3}
\lct(U_{\bar{r}+1},\Delta,\mathfrak{a})= \\
\min\left\lbrace s,
\widetilde{\Psi}_{\ell}(c,\alpha_{n_1},\ldots,\alpha_{n_{s}}) \mid
\ell \text{ such that }\beta_{1,\ell}=\cdots=\beta_{\bar{r},\ell}=0
\right\rbrace.
\end{multline}
\begin{multline} \label{EqUrr4}
\lct(U_{\bar{r}+1},\Delta,\mathfrak{b})= \\
\min\left\lbrace
\widetilde{\Psi}_{\ell}(c,\alpha_{n_1},\ldots,\alpha_{n_{s}},\alpha_{\bar{r}+1}) \mid
\ell \text{ such that }\beta_{1,\ell}=\cdots=\beta_{\bar{r},\ell}=0
\right\rbrace.
\end{multline}
Finally, for the case $\bar{r}<r$, the result follows from formulas (\ref{EqU1r}) and (\ref{EqUrr2}). For the case $\bar{r}=r$, it follows from formulas (\ref{EqU1r}), (\ref{EqUrr3}) and (\ref{EqUrr4}).
\end{proof}

\section{Computing lct for m-binomial ideals}
\label{SecMain}

Given a pseudo-resolution $\Sigma'\to\Sigma$ (\ref{DefPseudoRes}) of a m-binomial ideal (\ref{DefGBinId}) we may use Proposition~\ref{PropLctAfin} to compute the log-canonical threshold of the ideal, by computing the minimum of all affine charts in the pseudo-resolution.
However we will see that the computation of the pseudo-resolution may be avoided, instead of a regular subdivision $\Sigma'$ of the fan $\Sigma$ we only need a suitable subdivision $\Gamma$ of $\Sigma$ such that the minimum of a function is attained at some ray of $\Gamma$.

Given a $r\times n$ matrix $M$ with integer entries, we denote by $\row_i(M)\in\mathbb{Z}^n$ the $i$-th row of $M$.
\begin{Definition} \label{DefTriple}
We say that $(M^{+},M^{-},u)$ is a \emph{triple} if
\begin{itemize}
\item  $M^{+}$ and $M^{-}$ are $r\times n$ matrices with non-negative integer entries,

\item $u=(u_1,\ldots,u_r)\in k^{r}$ and

\item $u_i=0$ if and only if $\row_i(M^{+})=\row_i(M^{-})$.
\end{itemize}
\end{Definition}

\begin{Definition} \label{DefMatriz}
Let $W=\mathbb{A}_{k}^n=\Spec(k[x_1,\ldots,x_n])$ be the $n$-dimensional affine space.
Let $\mathfrak{a}\subset\OO_W$ be a m-binomial ideal (\ref{DefGBinId}) generated by $f_1,\ldots,f_r$. Where $f_i$ is either a monomial $x^{a_i}$ or a binomial $x^{a_i}-\tilde{u}_i x^{b_i}$, $i=1,\ldots,r$, $a_i,b_i\in\mathbb{Z}^n_{\geq 0}$ and $\tilde{u}_i\in k^{\ast}$.
We associate to the generators of the ideal $\mathfrak{a}$ a triple $(M^{+},M^{-},u)=(M^{+}_{\mathfrak{a}},M^{-}_{\mathfrak{a}},u_{\mathfrak{a}})$
(\ref{DefTriple}) as follows, for $i=1,\ldots,n$:
\begin{itemize}
\item if $f_i=x^{a_i}$ is a monomial then set $\row_i(M^{+})=\row_i(M^{-})=a_i$ and $u_i=0$, and

\item if $f_i=x^{a_i}-\tilde{u}_ix^{b_i}$ is a binomial then set
$\row_i(M^{+})=a_i$, $\row_i(M^{-})=b_i$ and $u_i=\tilde{u}_i$.
\end{itemize}
\end{Definition}

\begin{Definition} \label{DefLctM}
Let $(M^{+},M^{-},u)$ be a triple as in Definition \ref{DefTriple} and let
$c=(c_1,\ldots,c_n)\in\mathbb{Z}_{\geq 0}^n$ be a vector with non negative integer entries.
We define a function associated to the triple and $c$,
\begin{equation*}
\LCT(M^{+},M^{-},u,c):\mathbb{R}_{\geq 0}^n\to \mathbb{R}
\end{equation*}
as follows:
Let $v\in\mathbb{R}_{\geq 0}^n$ be a vector with non negative entries.
Set $M=M^{+}-M^{-}$,
\begin{itemize}
\item $\beta^{+}(v)=(\beta^{+}(v)_1,\ldots,\beta^{+}(v)_r)=M^{+}v$,

\item $\beta^{-}(v)=(\beta^{-}(v)_1,\ldots,\beta^{-}(v)_r)=M^{-}v$,

\item $\beta(v)=\beta^{+}(v)-\beta^{-}(v)=Mv$,

\item $\alpha(v)=\min\{\beta^{+}(v),\beta^{-}(v)\}=(\alpha(v)_1,\ldots,\alpha(v)_r)=$\newline
$\left(\min\{\beta^{+}(v)_1,\beta^{-}(v)_1\},\ldots,\min\{\beta^{+}(v)_r,\beta^{-}(v)_r\}\right)$.

\item Let $\varepsilon^{(v)}:\{1,\ldots,r\}\to\{1,\ldots,r\}$ be a permutation such that
\begin{equation*}
\alpha(v)_{\varepsilon^{(v)}(1)}\leq\alpha(v)_{\varepsilon^{(v)}(2)}\leq \cdots \leq \alpha(v)_{\varepsilon^{(v)}(r)}.
\end{equation*}

\item Let $r_0^{(v)}\leq r$ be the maximum index such that 
\begin{equation*}
\begin{array}{r}
\row_{\varepsilon^{(v)}(i)}(M)\neq 0 \\
\text{and }\beta_{\varepsilon^{(v)}(i)}=0
\end{array}
\qquad i=1,\ldots,r_0^{(v)}.
\end{equation*}

\item Let $n^{(v)}_1<n^{(v)}_2<\cdots <n^{(v)}_{s^{(v)}}\leq\bar{r}^{(v)}$ be the sequence associated to the first $r^{(v)}_0$ rows of $M$ and the first $r^{(v)}_0$ coordinates of the vector $u$ defined in Remark \ref{RemDefni},
but ordered with the permutation $\varepsilon^{(v)}$:
\begin{multline*} 
\mathcal{N}(\row_{\varepsilon^{(v)}(1)}(M),\ldots,\row_{\varepsilon^{(v)}(r^{(v)}_0)}(M),
u_{\varepsilon^{(v)}(1)},\ldots,u_{\varepsilon^{(v)}(r^{(v)}_0)})= \\
\left\lbrace
1=n^{(v)}_1<n^{(v)}_2<\cdots <n^{(v)}_{s^{(v)}}\leq\bar{r}^{(v)}\leq r^{(v)}_0
\right\rbrace.
\end{multline*}

Consider $c+\mathbf{1}=(c_1+1,\ldots,c_n+1)$ and the usual scalar product
$(c+\mathbf{1})\cdot v=(c_1+1)v_1+\cdots+(c_r+1)v_r$.
To simplify notation set $\varepsilon=\varepsilon^{(v)}$, $r_0=r^{(v)}_0$, $s=s^{(v)}$, $\bar{r}=\bar{r}^{(v)}$ and
$n_i=n^{(v)}_i$ for $i=1,\ldots,s$.

\item If $\bar{r}^{(v)}<r$ then set
\begin{gather*} 
\LCT(M^{+},M^{-},u,c)(v) = 
\widetilde{\Psi}\left((c+\mathbf{1})\cdot v-1,
\alpha(v)_{\varepsilon(n_1)},\ldots,\alpha(v)_{\varepsilon(n_{s})},\alpha(v)_{\varepsilon(\bar{r}+1)}\right)\\
\begin{aligned}
\hspace{2.5cm} & = \min\left\lbrace
\frac{(c+\mathbf{1})\cdot v}{\alpha(v)_{\varepsilon(n_1)}},\right. \\
  & \frac{(c+\mathbf{1})\cdot v+
\left(\alpha(v)_{\varepsilon(n_2)}-\alpha(v)_{\varepsilon(n_{1})}\right)}{\alpha(v)_{\varepsilon(n_{2})}}, \\
  & \ldots \ldots \ldots,\\
  & \frac{
(c+\mathbf{1})\cdot v+
\left(\alpha(v)_{\varepsilon(n_{s})}-\alpha(v)_{\varepsilon(n_{s-1})}\right)+\cdots+
\left(\alpha(v)_{\varepsilon(n_{s})}-\alpha(v)_{\varepsilon(n_1)}\right)}{\alpha(v)_{\varepsilon(n_{s})}}, \\
  & \left.\frac{
(c+\mathbf{1})\cdot v+
\left(\alpha(v)_{\varepsilon(\bar{r}+1)}-\alpha(v)_{\varepsilon(n_{s})}\right)+\cdots+
\left(\alpha(v)_{\varepsilon(\bar{r}+1)}-\alpha(v)_{\varepsilon(n_1)}\right)}{\alpha(v)_{\varepsilon(\bar{r}+1)}}
\right\rbrace.
\end{aligned}
\end{gather*}

\item If $\bar{r}^{(v)}=r$ then set 
\begin{equation*}
\LCT(M^{+},M^{-},u,c)(v)=\min\left\lbrace s,
\widetilde{\Psi}((c+\mathbf{1})\cdot v-1,
\alpha(v)_{\varepsilon(n_1)},\ldots,\alpha(v)_{\varepsilon(n_s)})
\right\rbrace
\end{equation*}
\end{itemize}
We also define 
\begin{equation*}
\LCT^{\ast}(M^{+},M^{-},u,c)(v)=\dfrac{(c+\mathbf{1})\cdot v}{\alpha(v)_{\varepsilon(n_1)}}.
\end{equation*}
\end{Definition}
Note that the definition of $\widetilde{\Psi}$ in \ref{DefLctM} is consistent with Definition \ref{DefPsi}.
If the vector $v$ has integer entries then all the setting in Definition~\ref{DefLctM} represents the exponents of vertex $v$ in a subdivision.
The index $r_0$ represents the last binomial generator.

\begin{Remark}
Note that $\LCT^{\ast}(M^{+},M^{-},u,c)=\LCT(N,N,0,c)$, where $N$ is the $2r\times n$ matrix obtained by joining the rows of $M^{+}$ and $M^{-}$. 
In terms of ideals, if the triple $(M^{+},M^{-},u)$ corresponds to a m-binomial ideal $\mathfrak{a}$ then the triple $(N,N,0)$ corresponds to the monomial ideal generated by all monomials appearing in the generators of $\mathfrak{a}$.
Note also that in Definition~\ref{DefLctM} the coefficients $u$ are used only to set the index $\bar{r}$.
\end{Remark}

\begin{Example}
Set $W=\Spec(k[x_1,x_2,x_3,x_4])$ ant let $\mathfrak{a}\subset k[x_1,x_2,x_3,x_4]$ be the ideal defining the monomial curve $(t^6,t^8,t^{10},t^{11})$. Generators of $\mathfrak{a}$ may be computed with Singular \cite{Singular2012}:
\begin{equation*}
\mathfrak{a}=\left\langle
x_{2}^{2}-x_{1}x_{3},
x_{1}^{3}-x_{2}x_{3},
x_{1}^{2}x_{2}-x_{3}^{2},
x_{1}^{2}x_{3}-x_{4}^{2}
\right\rangle.
\end{equation*}
The triple associated to $\mathfrak{a}$ is $(M^{+},M^{-},u)$ where
\begin{equation*}
M^{+}=
\left(
\begin{array}{*{4}{r}}
0 & 2 & 0 & 0\\
3 & 0 & 0 & 0\\
2 & 1 & 0 & 0\\
2 & 0 & 1 & 0
\end{array}
\right),
\qquad
M^{-}=
\left(
\begin{array}{*{4}{r}}
1 & 0 & 1 & 0\\
0 & 1 & 1 & 0\\
0 & 0 & 2 & 0\\
0 & 0 & 0 & 2
\end{array}
\right),
\qquad
u=(1,1,1,1).
\end{equation*}
We compute the value $\LCT(M^{+},M^{-},u,0)(v)$ for $v=(6,8,10,11)$. The ingredients in Definition~\ref{DefLctM} are the following:

\begin{gather*}
\beta^{+}=\beta_{v}^{+}=M^{+}v=(16,18,20,22),
\qquad
\beta^{-}=\beta_{v}^{-}=M^{-}v=(16,18,20,22),\\
\beta=\beta_v=(M^{+}-M^{-})v=(0,0,0,0),
\qquad
\alpha=\alpha_v=\min\{M^{+}v,M^{-}v\}=(16,18,20,22),
\end{gather*}
The permutation $\varepsilon$ is the identity.
The number $r_0=r=4$ since all generators are binomials and $\beta=0$.
Since all coordinates of the vector $u$ are $1$ then all rows of $M$ are compatible (\ref{DefBetaUComp}) and then $\bar{r}=r_0=r=4$.

The sequence $\mathcal{N}(\row_1(M),\row_2(M),\row_3(M),\row_4(M),1,1,1,1)$ (\ref{RemDefni}) is
$n_1=1<n_2=2<n_3=4=\bar{r}$ and $s=3$, the rank of $M$.
Note that the third row of the matrix $M$ is in the linear span of the first two rows.

Finally $\LCT(M^{+},M^{-},u,0)(v)$ is the minimum of
\begin{eqnarray*}
s & = & 3, \\
\frac{\mathbf{1}\cdot v}{\alpha_{\varepsilon(n_1)}} & = & \dfrac{35}{16}, \\
\frac{\mathbf{1}\cdot v+(\alpha_{\varepsilon(n_2)}-\alpha_{\varepsilon(n_{1})})}{\alpha_{\varepsilon(n_{2})}} &
= & \frac{35+2}{18}=\dfrac{37}{18}, \\
\frac{
\mathbf{1}\cdot v+
(\alpha_{\varepsilon(n_{3})}-\alpha_{\varepsilon(n_{2}})+
(\alpha_{\varepsilon(n_{3})}-\alpha_{\varepsilon(n_1)})}{\alpha_{\varepsilon(n_{3})}} & = &
\frac{35+4+6}{22}=\dfrac{45}{22},
\end{eqnarray*}
so that $\LCT(M^{+},M^{-},u,0)(v)=45/22$.

Now consider the ideal
\begin{equation*}
\mathfrak{a}_1=\left\langle
x_{2}^{2}-x_{1}x_{3},
x_{1}^{3}-x_{2}x_{3},
x_{1}^{2}x_{2}+x_{3}^{2},
x_{1}^{2}x_{3}-x_{4}^{2}
\right\rangle.
\end{equation*}
Note that the generators of $\mathfrak{a}_1$ only differ in one coefficient from those of $\mathfrak{a}$. The triple associated to $\mathfrak{a}_1$ is $(M^{+},M^{-},u_1)$ with the same matrices and $u_1=(1,1,-1,1)$.

Here we have the same $\beta=(0,0,0,0)$, $\alpha=(16,18,20,22)$ and $r_0=r=4$.
But now $\row_1(M),\row_2(M),\row_3(M)$ and $(1,1,-1)$ are not compatible, see (\ref{DefBetaUComp}).
The sequence $\mathcal{N}(\row_1(M),\row_2(M),1,1)$ is
$n_1=1<n_2=2=\bar{r}$.
In this case $\LCT(M^{+},M^{-},u,0)(v)$ is the minimum of
\begin{eqnarray*}
\frac{\mathbf{1}\cdot v}{\alpha_{\varepsilon(n_1)}} & = & \dfrac{35}{16}, \\
\frac{\mathbf{1}\cdot v+(\alpha_{\varepsilon(n_2)}-\alpha_{\varepsilon(n_{1})})}{\alpha_{\varepsilon(n_{2})}} &
= & \frac{35+2}{18}=\dfrac{37}{18}, \\
\frac{
\mathbf{1}\cdot v+
(\alpha_{\varepsilon(\bar{r}+1)}-\alpha_{\varepsilon(n_{2}})+
(\alpha_{\varepsilon(\bar{r}+1)}-\alpha_{\varepsilon(n_1)})}{\alpha_{\varepsilon(\bar{r}+1)}} & = &
\frac{35+2+4}{20}=\dfrac{41}{20}.
\end{eqnarray*}
We conclude that $\LCT(M^{+},M^{-},u_1,0)(v)=41/20$
\end{Example}

\begin{Lemma} \label{LemProyLct}
The value of $\LCT(M^{+},M^{-},u,c)(v)$ only depends on the ray defining $v$.
If $v\in\mathbb{R}_{\geq 0}^n$ then
$\LCT(M^{+},M^{-},u,c)(v)=\LCT(M^{+},M^{-},u,c)(\lambda v)$, for every $\lambda>0$.
And the same occurs with $\LCT^{\ast}$.
\end{Lemma}

\begin{proof}
It follows from Definition \ref{DefLctM}.
\end{proof}
This means that the map $\LCT(M^{+},M^{-},u,c)$ is, in some sense, projective.

\begin{Theorem} \label{ThResolLCT}
Let $W=U_{\tau}$ be a $n$-dimensional affine smooth toric variety associated to a cone $\tau$.
Set $\Sigma$ to be the fan associated to $\tau$.
Consider a set of vertices $\Xi=\{v_1,\ldots,v_n\}$ extending the rays of $\tau$.
Let $\mathfrak{a}\subset\OO_W$ be a m-binomial ideal (\ref{DefGBinId}) generated by $f_1,\ldots,f_r$, where each $f_i$ is either a monomial or a binomial. Let $(M^{+},M^{-},u)$ be the triple associated to $f_1,\ldots,f_r$ (\ref{DefMatriz}) and
$v^{c}=v_1^{c_1}\cdots v_n^{c_n}$ be the monomial defined by a divisor $\Delta$ in $W$.

If $\Pi:W'\to W$ is a pseudo resolution associated to a regular subdivision $\Sigma'\supset \Sigma$ then
\begin{equation*}
\lct(W,\Delta,\mathfrak{a})=\min\left\lbrace
\LCT(M^{+},M^{-},u,c)(v) \mid v\in\Sigma'(1)
\right\rbrace
\end{equation*}
\end{Theorem}

\begin{proof}
Set $\Delta'=\Pi^{\ast}-K_{W'/W}$ and let $\sigma\in\Sigma'$ be a cone.
By Proposition~\ref{PropLctCover}, it is enough to prove that
\begin{equation*}
\lct(U_{\sigma},\Delta',\mathfrak{a}^{\ast})=
\min\left\lbrace
\LCT(M^{+},M^{-},u,c)(v) \mid v\in\sigma(1)
\right\rbrace
\end{equation*}

We may assume, for simplicity, that $v_1,\ldots,v_n$ are the canonical vectors
$$v_i=e_i=(0,\ldots,0,\stackrel{i}{1},0,\ldots,0) \qquad i=1,\ldots,n.$$
We extend the $r\times n$ matrices $M^{+}$ and $M^{-}$ to $r\times m$ matrices $\widetilde{M}^{+}$ and $\widetilde{M}^{-}$
\begin{gather*}
\widetilde{M}^{+}=\left(
\begin{array}{cccccc}
a_{1,1} & \cdots & a_{1,n} & a_{1,n+1} & \cdots & a_{1,m} \\
\vdots & & \vdots & \vdots & & \vdots \\
a_{r,1} & \cdots & a_{r,n} & a_{r,n+1} & \cdots & a_{r,m}
\end{array}
\right) \\
\widetilde{M}^{-}=\left(
\begin{array}{cccccc}
b_{1,1} & \cdots & b_{1,n} & b_{1,n+1} & \cdots & b_{1,m} \\
\vdots & & \vdots & \vdots & & \vdots \\
b_{r,1} & \cdots & b_{r,n} & b_{r,n+1} & \cdots & b_{r,m}
\end{array}
\right)
\end{gather*}
such that they represent the total transforms of $f_1,\ldots,f_r$ in $W'$.

We consider the total coordinate ring $k[v_1,\ldots,v_m]$ of $W_{\Sigma'}$ (\ref{RemTotalCRing}).
For each $i=1,\ldots,r$, if $f_i$ is a monomial then the
\begin{equation*}
f_i^{\ast}=v^{a_i}=v_1^{a_{i,1}}\cdots v_m^{a_{i,m}}.
\end{equation*}
If $f_i$ is a binomial then
\begin{equation*}
f_i^{\ast}=v^{a_i}-v^{b_i}=
v_1^{a_{i,1}}\cdots v_m^{a_{i,m}}-u_iv_1^{b_{i,1}}\cdots v_m^{b_{i,m}}.
\end{equation*}
For $j=n+1,\ldots,m$ the $j$-th column of matrix $\widetilde{M}^{\pm}$ is obtained as follows:
Since $v_j$ is a linear combination of $v_1,\ldots,v_n$, say
\begin{equation*}
v_j=\sum_{\ell=1}^{n}\lambda_{\ell,j}v_{\ell}
\end{equation*}
then we have that
\begin{equation} \label{EqCompatibleAB}
a_{i,j}=\sum_{\ell=1}^{n}a_{i,\ell}\lambda_{\ell,j},
\qquad
b_{i,j}=\sum_{\ell=1}^{n}b_{i,\ell}\lambda_{\ell,j}.
\end{equation}
See Remark \ref{RemComplex}, Remark \ref{RemTotalCRing} and Definition \ref{DefGlobalMon}.
 
Consider the relative canonical sheaf $K_{W'/W}$, let $v^{\kappa}=v_1^{\kappa_1}\cdots v_m^{\kappa_m}$ the monomial generating the ideal $\OO_{W'}(-K_{W'/W})$. It is well known that
\begin{equation*}
\kappa_j+1=\mathbf{1}\cdot v_j=\sum_{\ell=1}^{n}\lambda_{\ell,j}
\end{equation*}
The monomial associated to the divisor $\Delta'$ is
$v_1^{c_1}\cdots v_m^{c_m}$ where
\begin{equation*}
c_j+1=\sum_{\ell}^{n}(c_{\ell}+1)\lambda_{\ell,j}
\end{equation*}
Let $v_{j_1},\ldots,v_{j_n}$ be vertices in $\Xi'$ such that they are a basis of the lattice
$N\equiv\mathbb{Z}^n$ and the rays of $\sigma$ are $v_{j_1},\ldots,v_{j_{n_0}}$.
For $i=1,\ldots,r$ the total transform of $f_i$ at the open set $U_{\sigma}$ is
\begin{gather*}
v_{j_1}^{a_{i,j_1}}\cdots v_{j_n}^{a_{i,j_n}}
\qquad\text{if } f_i \text{ is a monomial}, \\
v_{j_1}^{a_{i,j_1}}\cdots v_{j_n}^{a_{i,j_n}}-
u_i v_{j_1}^{b_{i,j_1}}\cdots v_{j_n}^{b_{i,j_n}}
\qquad\text{if } f_i \text{ is a binomial}.
\end{gather*}
Since $W'\to W$ is a pseudo-resolution, in the case of $f_i$ being a binomial we have
\begin{equation*}
v_{j_1}^{a_{i,j_1}}\cdots v_{j_n}^{a_{i,j_n}}-
u_i v_{j_1}^{b_{i,j_1}}\cdots v_{j_n}^{b_{i,j_n}}=
\pm
v_{j_1}^{\alpha_{i,j_1}}\cdots v_{j_n}^{\alpha_{i,j_n}}
\left(
1-u_i^{\pm} v_{j_1}^{\beta_{i,j_1}}\cdots v_{j_n}^{\beta_{i,j_n}}
\right)
\end{equation*}
Where $\beta_{i,j_{\ell}}=|a_{i,j_{\ell}}-b_{i,j_{\ell}}|$ and
$\alpha_{i,j_{\ell}}=\min\{a_{i,j_{\ell}},b_{i,j_{\ell}}\}$, for $\ell=1,\ldots,n$.

Set $\alpha_i=(\alpha_{i,j_1},\ldots,\alpha_{i,j_n})$ and
$\beta_i=(\beta_{i,j_1},\ldots,\beta_{i,j_n})$ for $i=1,\ldots,r$ and set
$x_j=v_{i_j}$ for $j=1,\ldots,n$.
We may reorder the generators $f_1,\ldots,f_r$ such that, for this cone $\sigma$, we have
that the ideal $\mathfrak{a}\OO_{U_{\sigma}}$ is either
\begin{equation*}
\left\langle
x^{\alpha_1}(1-u_1 x^{\beta_1}),\ldots,x^{\alpha_r}(1-u_r x^{\beta_r}
\right\rangle
\end{equation*}
or
\begin{equation*}
\left\langle
x^{\alpha_1}(1-u_1 x^{\beta_1}),\ldots,x^{\alpha_{r'}}(1-u_{r'} x^{\beta_{r'}},
x^{\alpha_{r'+1}}
\right\rangle
\end{equation*}
for some $r'<r$. Where in both cases
$\alpha_1\preceq \alpha_2\preceq \cdots\preceq\alpha_r$.
So that we are in the setting of Proposition~\ref{PropLctAfin}.

Note first that the permutation $\varepsilon$ in Definition~\ref{DefLctM} for all rays $v$ of $\sigma$ is the identity.
Note also that since $v_1,\ldots,v_n$ and $v_{i_1},\ldots,v_{i_n}$ are both basis, it follows from
(\ref{EqCompatibleAB}) that
the sequence
$\mathcal{N}(\beta_1,\ldots,\beta_r,u)=\{n_1<n_2\cdots<n_s<\bar{r}\}$
is the same as the sequence
$\mathcal{N}(\row_1(M),\ldots,\row_r(M),u)$, where $M=M^{+}-M^{-}$.

Now the result follows by comparing Definition~\ref{DefLctM} for every ray $v_{i_{\ell}}$ of $\sigma$ and Proposition~\ref{PropLctAfin} in terms of functions $\Psi_{\ell}$.
\end{proof}

\begin{Remark}
Let $W=U_{\tau}$ be a $n$-dimensional affine smooth toric variety associated to a cone $\tau$ as in Theorem \ref{ThResolLCT}, fix a point $\xi\in U_{\tau}$ in the orbit of the distinguished point of $U_{\tau}$ (see \cite[page 116]{CoxLittleSchenck2011}) and assume that the m-binomial ideal $\mathfrak{a}\subset\OO_{U_{\tau}}$ is such that $\mathfrak{a}_{\xi}\neq\OO_{U_{\tau},\xi}$, then
\begin{equation*}
\lct(U_{\tau},\Delta,\mathfrak{a})_{\xi}=
\lct(U_{\tau},\Delta,\mathfrak{a}).
\end{equation*}
 
If $W=W_{\Sigma}$ is a smooth toric variety and $\mathfrak{a}\subset\OO_W$ is a m-binomial ideal (\ref{DefGBinId}) then
\begin{equation*}
\lct(W,\Delta,\mathfrak{a})=
\min\left\{\lct(U_{\sigma},\Delta|_{U_{\sigma}},\mathfrak{a}|_{U_{\sigma}}) \mid \sigma\in\Sigma\right\}.
\end{equation*}
For a point $\xi\in W$ such that $\mathfrak{a}_{\xi}\neq\OO_{W,\xi}$,
set $\sigma\in\Sigma$ the unique cone such that $\xi\in U_{\sigma}$ and $\xi$ is in the orbit of the distinguished point of $U_{\sigma}$ (see \cite[page 116]{CoxLittleSchenck2011}),
we have that
\begin{equation*}
\lct(W,\Delta,\mathfrak{a})_{\xi}=
\lct(U_{\tau},\Delta|_{U_{\tau}},\mathfrak{a}|_{U_{\tau}}).
\end{equation*}
\end{Remark}

Theorem~\ref{ThResolLCT} gives a way of computing the log-canonical threshold of a m-binomial ideal, but one needs to compute a pseudo-resolution of the ideal. We want to avoid this computation and express the log-canonical threshold of the ideal in terms of simpler computations.

The function $\LCT(M^{+},M^{-},u,c)$ is not continuous in general, but we may stratify the space $\mathbb{R}^{n}_{\geq 0}$ such that the function $\LCT(M^{+},M^{-},u,c)$ is continuous in every stratum. In fact this stratification is given by a fan.
\begin{Proposition} \label{PropStrata}
Let $(M^{+},M^{-},u)$ be a triple (\ref{DefTriple}).

If $c\in\mathbb{Z}^n$ then 
there is a fan $\Gamma$, with support $\mathrm{R}^{n}_{\geq 0}$, such that for every cone $\gamma\in\Gamma$ the function
$\LCT(M^{+},M^{-},u,c)$ is continuous in the relative interior of $\gamma$.
\end{Proposition}

\begin{proof}
Consider the hyperplanes of $\mathbb{R}^n$ having normal vectors:
\begin{equation} \label{EqHPlanes}
\begin{gathered}
\row_i(M^{+})-\row_i(M^{-}) \qquad i=1,\ldots,n\ , \\
\row_i(M^{\pm})-\row_j(M^{\pm}) \qquad i<j.
\end{gathered}
\end{equation}
Let $\Gamma$ be the fan obtained by the subdivision of $\mathbb{R}^{n}_{\geq 0}$ given by these hyperplanes.
For every cone $\gamma\in\Gamma$ the relative interior of $\gamma$ is defined by some hyperplanes equalities and some hyperplane inequalities, $> 0$ or $< 0$ of (\ref{EqHPlanes}).

By our construction, the permutation $\varepsilon^{(v)}$, the number $r^{(v)}_o$ and the sequence $n^{(v)}_1<\cdots <n^{(v)}_{s^{(v)}}\leq\bar{r}^{(v)}$ in Definition~\ref{DefLctM} are the same for all $v$ in the relative interior of $\gamma$.
The function $\LCT(M^{+},M^{-},u,c)$ in the relative interior of $\gamma$ is continuous since it is the minimum of continuous functions.
\end{proof}

Note that our fan $\Gamma$ need not to be regular. The fan used for embedded resolution in
\cite{GonzalezTeissier2002} is compatible with a subset of hyperplanes in (\ref{EqHPlanes}),
see also \cite[Proposition~6.3]{Teissier2004}.

In fact, we prove that the minimum of the function $\LCT(M^{+},M^{-},u,c)$ exists and it is equal to $\lct(W,\Delta,\mathfrak{a})$.
\begin{Theorem} \label{ThlctMinv}
Let $W=U_{\tau}$ be a $n$-dimensional affine smooth toric variety associated to a cone $\tau\subset\mathbb{R}_{\geq 0}^{n}$.
Set $\Sigma$ to be the fan associated to $\tau$.
Consider a set of vertices $\Xi=\{v_1,\ldots,v_n\}$ extending the rays of $\tau$.
Let $\mathfrak{a}\subset\OO_W$ be a m-binomial ideal (\ref{DefGBinId}) generated by $f_1,\ldots,f_r$, where each $f_i$ is either a monomial or a binomial. Let $(M^{+},M^{-},u)$ be the triple associated to $f_1,\ldots,f_r$ (\ref{DefMatriz}) and
$v^{c}=v_1^{c_1}\cdots v_n^{c_n}$ be the monomial defined by a divisor $\Delta$ in $W$.

Then the $lct$ of the ideal $\mathfrak{a}$ is the minimum of the function 
$\LCT(M^{+},M^{-},u,c)$:
\begin{equation*}
\lct(W,\Delta,\mathfrak{a})=\min\left\{\LCT(M^{+},M^{-},u,c)(v) \mid v\in\tau\right\}.
\end{equation*}
\end{Theorem}

\begin{proof}
First note that $\LCT_{(M^{+},M^{-},u,c)}(v)\geq 0$ so that $\inf\{\LCT(M^{+},M^{-},u,c)(v)\mid v\in\mathbb{R}^{n}_{\geq 0}\}$ exists.

By Theorem \ref{ThResolLCT} we have
\begin{equation*}
\lct(W,\Delta,\mathfrak{a})\geq \inf\left\{\LCT(M^{+},M^{-},u,c)(v) \mid v\in\mathbb{R}^{n}_{\geq 0}\right\}.
\end{equation*}

We only have to prove the reverse inequality. We want to prove that for any $v\in\tau$ we have
\begin{equation} \label{EqIneqVQ}
    \lct(W,\Delta,\mathfrak{a})\leq \LCT(M^{+},M^{-},u,c)(v).
\end{equation}

If $v\in\tau\cap\mathbb{Q}^{n}_{\geq 0}$, let $m\in\mathbb{N}$ such that 
$mv\in\mathbb{N}^{n}$. Consider a subdivision $\Sigma_1$ of $\Sigma$ containing 
$mv$ and refine $\Sigma_1$ in order to have a subdivision 
$\Sigma'$ and a pseudo-resolution of $\mathfrak{a}$ with $mv$ a ray of $\Sigma'$. 
By Theorem \ref{ThResolLCT} and Lemma \ref{LemProyLct} we have inequality (\ref{EqIneqVQ}).

Let $\Gamma$ be the fan of Proposition~\ref{PropStrata}.
If $v\in\Gamma(1)$ then there exist $\lambda\in\mathbb{R}$ with $\lambda v\in\mathbb{Q}^{n}_{\geq 0}$ and again by Lemma~\ref{LemProyLct} we obtain inequality (\ref{EqIneqVQ}).

Assume that $v\in\tau$ and $v\not\in\Gamma(1)$, there exists a unique cone $\gamma\in\Gamma$ with $v$ in the relative interior of $\gamma$.
There is a sequence $\left\{w_{\ell}\right\}_{\ell\in\mathbb{N}}$ such that
 with , with 
\begin{itemize}
\item $w_{\ell}\in\mathbb{Q}^{n}_{\geq 0}$ for all $\ell\in\mathbb{N}$,

\item $w_{\ell}$ is in the relative interior of $\gamma$ and

\item  $\lim_{\ell\to\infty}w_{\ell}=v$.
\end{itemize}
We have seen that 
\begin{equation*}
\lct(W,\Delta,\mathfrak{a})\leq\LCT(M^{+},M^{-},u,c)(w_{\ell})
\qquad \forall \ell\in\mathbb{N}.
\end{equation*}
Finally, it follows from Proposition~\ref{PropStrata} that inequality (\ref{EqIneqVQ}) holds for $v$.
\end{proof}

By Theorem \ref{ThlctMinv}, the problem of computing $\lct(W,\Delta,\mathfrak{a})$ is reduced to the problem of computing the minimum of the function $\LCT(M^{+},M^{-},u,c)$. The last problem relies on computing the rays of the fan determined by the rows of the matrices $M^{+}$ and $M^{-}$.

\begin{Theorem} \label{ProplctRay}
Let $(M^{+},M^{-},u)$ be the triple associated to $\mathfrak{a}$ and let $\Gamma$ be the fan given by Proposition \ref{PropStrata}.

The minimum of the function $\LCT(M^{+},M^{-},u,c)$ (\ref{DefLctM}) is attained at some ray of the fan $\Gamma$.
If $\Gamma(1)=\{w_1,\ldots,w_t\}$ then
\begin{equation*}
\min\LCT(M^{+},M^{-},u,c)=\min\left\{ \LCT(M^{+},M^{-},u)(w_i) \mid i=1,\ldots,t\right\}.
\end{equation*}
\end{Theorem}

\begin{proof}
Let $\gamma\in\Gamma$ be a cone. As in the proof of Proposition \ref{PropStrata} There is a permutation $\varepsilon$ of $\{1,\ldots,r\}$ such that, with the notation of Definition \ref{DefLctM},
\begin{equation*}
\alpha_{v,\varepsilon(1)}\leq \alpha_{v,\varepsilon(2)}\leq \cdots \leq \alpha_{v,\varepsilon(r)}
\qquad \forall v\in\gamma.
\end{equation*}
For $v$ in the relative interior of $\gamma$ set $\varphi(v)=\LCT(M^{+},M^{-},u,c)(v)$.
Consider the function $\varphi:\gamma\to\mathbb{R}$ extended by continuity to the boundary.
\begin{description}
\item[Claim 1] If $v\in\gamma$ is such that $\varphi(v)=\inf\varphi$ then $v$ is in the relative boundary of $\gamma$.

\item[Claim 2] If $v$ is in the relative boundary of $\gamma$ then $\LCT(M^{+},M^{-},u)(v)\leq\varphi(v)$.
\end{description}
First note that proposition will follow from claim~1 and claim~2.

For claim~1, note that $\varphi$ is the minimum of functions like
\begin{equation*}
\varphi_i(v)=\frac{
(c+\mathbf{1})\cdot v+
(\alpha_{v,\varepsilon(n_{i})}-\alpha_{v,\varepsilon(n_{i-1}})+\cdots+
(\alpha_{v,\varepsilon(n_{i})}-\alpha_{v,\varepsilon(n_1)})}{\alpha_{v,\varepsilon(n_{i})}}
\end{equation*}
for $i=1,\ldots,s_v$,  see Definition~\ref{DefLctM}.
Since the $\alpha_{v,\varepsilon(n_j)}$ are linear functions on $v\in\gamma$ and $\varphi_i(\lambda v)=\varphi_i(v)$ (\ref{LemProyLct}) we have that
\begin{equation*}
\inf\left\{\varphi_i(v)\mid v\in\gamma\right\}=
\inf\left\{\varphi_i(v) \mid \alpha_{v,\varepsilon(n_i)}=1,\ v\in\gamma\right\}.
\end{equation*}
The function $\varphi_i$ restricted to the hyperplane $\alpha_{v,\varepsilon(n_i)}=1$ is a linear function and its minimum will be attained at some point of the boundary of the domain. This proves claim~1.

For claim~2 note that the number $s_v$ is constant for all $v$ in the relative interior of $\gamma$.
If $v'$ is in the relative boundary of $\gamma$ then $s_{v'}\geq s_v$ and claim~2 follows.
\end{proof}

\section{Computation and examples} \label{SecExample}

This section is devoted to show several examples of computation of $\lct$ of m-binomial ideals.
All computations have been made with Singular \cite{Singular2012}.
We describe a simple procedure to compute the rays of the fan $\Gamma$ appearing in Theorem~\ref{ProplctRay}. It is a naive procedure, but even if it may be improved it avoids any computation of a pseudo-resolution of the ideal.
In fact, complexity is bounded only in terms of the number of variables $n$ and the number of generators $r$.
\begin{Remark}
Let $\mathfrak{a}\subset k[x_1,\ldots,x_n]$ be a m-binomial ideal generated by $f_1,\ldots,f_r$, where each $f_i$ is either a monomial or a binomial.
Set $(M^{+},M^{-},u)$ be the triple associated to $f_1,\ldots,f_r$ (\ref{DefMatriz}).
We denote $M_i^{\pm}$ the $i$-th row of the matrix $M^{\pm}$.

By Theorem~\ref{ProplctRay} we shall check the minimum of the values $\LCT(M^{+},M^{-},u,0)(v)$ for every $v\in\Gamma(1)$, where $\Gamma$ is the fan given by Proposition \ref{PropStrata}.

Set $A$ to be the identity $n\times n$ matrix.
Following proof of Proposition \ref{PropStrata}, add to the matrix $A$ the rows
\begin{gather*}
M^{+}_i-M^{-}_i \qquad i=1,\ldots,n\ , \\
M^{\pm}_i-M^{\pm}_j \qquad i<j.
\end{gather*}
We may delete some rows of $A$:
\begin{itemize}
\item every row with $A_i=0$,

\item every row $A_i$ with non-negative entries and

\item every row $A_i$ proportional to some other row $A_j$.
\end{itemize}

Every ray in $\Gamma(1)$ is defined by $n-1$ $\mathbb{Q}$-linear independent rows of the matrix $A$.
The problem of enumerating all rays of a fan has been solved in \cite{AvisFukuda1992}, the authors proved that the full list of vertices can be found in $O(r_{A}^2(n-1)N_0)$ time, where $N_0$ is the number of rays and $r_A$ is the number of rows of the matrix $A$.
Note that $r_A\leq n+\binom{2r}{2}$ and $N_0\leq \binom{r_A}{n-1}$. In general the problem of estimating $N_0$ is hard.
Once we have obtained the list of all rays, then we shall evaluate the function $\LCT(M^{+},M^{-},u,0)$ at every ray and compute the minimum.
\end{Remark}

\begin{figure}[t]
\begin{equation*}
\begin{array}{|c|r|l|r|l|}
\hline
\text{Ray} & \multicolumn{2}{c|}{\LCT} & \multicolumn{2}{c|}{\LCT^{\ast}} \\ \hline
(2,1,3) & 3 & 3 & 3 & 3\\
(2,2,3) & 7/4 & 1.75 & 7/4 & 1.75\\
(1,1,2) & 2 & 2 & 2 & 2\\
(4,5,6) & 16/11 & 1.4545 & 3/2 & 1.5\\
\mathbf{(3,4,5)} & \mathbf{13/9} & \mathbf{1.4444} & 3/2 & 1.5\\
(1,1,1) & 3/2 & 1.5 & 3/2 & 1.5\\
(2,3,2) & 7/4 & 1.75 & 7/4 & 1.75\\
(4,6,7) & 17/11 & 1.5454 & 17/11 & 1.5454\\
(2,3,5) & 5/3 & 1.6667 & 5/3 & 1.6667\\
(2,3,3) & 8/5 & 1.6 & 8/5 & 1.6\\
(2,3,4) & 3/2 & 1.5 & 3/2 & 1.5\\
(1,2,1) & 2 & 2 & 2 & 2\\
(2,4,3) & 9/5 & 1.8 & 9/5 & 1.8\\
(1,2,2) & 5/3 & 1.6667 & 5/3 & 1.6667\\
(1,2,3) & 2 & 2 & 2 & 2\\
\hline
\end{array}\
\begin{array}{|c|r|l|r|l|}
\hline
\text{Ray} & \multicolumn{2}{c|}{\LCT} & \multicolumn{2}{c|}{\LCT^{\ast}} \\ \hline
(1,3,0) & 2 & 2 & \infty & \infty \\
(1,1,0) & 2 & 2 & \infty & \infty \\
(2,1,0) & 2 & 2 & \infty & \infty \\
(2,3,0) & 2 & 2 & \infty & \infty \\
(1,2,0) & 2 & 2 & \infty & \infty \\
(1,0,3) & 2 & 2 & \infty & \infty \\
(2,0,3) & 2 & 2 & \infty & \infty \\
(1,0,2) & 2 & 2 & \infty & \infty \\
(1,0,1) & 2 & 2 & \infty & \infty \\
(1,0,0) & 2 & 2 & \infty & \infty \\
(0,2,1) & 2 & 2 & \infty & \infty \\
(0,1,1) & 2 & 2 & \infty & \infty \\
(0,1,2) & 2 & 2 & \infty & \infty \\
(0,1,0) & 2 & 2 & \infty & \infty \\
(0,0,1) & 2 & 2 & \infty & \infty \\
\hline \end{array}
\end{equation*}
\caption{Rays for curve $(t^3,t^4,t^5)$}\label{Fig345}
\end{figure}

\begin{Example}
Let $V\subset\mathbb{A}^{3}$ be the monomial curve given by the parametrization $(t^3,t^4,t^5)$. The ideal $\mathfrak{a}$, defining $V$, is generated by the binomials:
\begin{equation*}
\left\langle x_{2}^{2}-x_{1}x_{3},
x_{2}x_{3}-x_{1}^{3},
x_{3}^{2}-x_{1}^{2}x_{2} \right\rangle.
\end{equation*}
We have obtained these generators with Singular \cite{Singular2012}.
In fact these three binomials form a standard basis of the ideal $\mathfrak{a}$.
The triple $(M^{+}$, $M^{-},u)$ associated to these generators is
\begin{equation*}
M^{+}=\left(
\begin{array}{*{3}{r}}
0 & 2 & 0\\
0 & 1 & 1\\
0 & 0 & 2
\end{array}
\right)
\qquad
M^{-}=\left(
\begin{array}{*{3}{r}}
1 & 0 & 1\\
3 & 0 & 0\\
2 & 1 & 0
\end{array}
\right)
\qquad
u=(1,1,1).
\end{equation*}

The transpose of matrix $A$, after deleting superfluous rows, is
\begin{equation*}
A^t=\left(
\begin{array}{*{14}{c}}
1 & 0 & 0 & -2 & -3 & -1 & 0 & -1 & 1 & 2 & 2 & 3 & 3 & 1 \\
0 & 1 & 0 & 1 & 2 & 2 & 2 & 1 & 1 & 1 & 0 & 0 & -1 & 0 \\
0 & 0 & 1 & 0 & 0 & -1 & -2 & 0 & -1 & -2 & -1 & -2 & -1 & -1
\end{array}
\right).
\end{equation*}
We obtain $30$ rays and the minimum of $\LCT(M^{+},M^{-},u,0)$ is $\lct(\mathfrak{a})=13/9$.
Figure~\ref{Fig345} is the list of all the rays of $\Gamma$, each vector is the solution of the linear system obtained by choosing two rows of $A$.
We also show the values of $\LCT(M^{+},M^{-},u,0)$ and  $\LCT^{\ast}(M^{+},M^{-},u,0)$.

The ray giving the minimum coincides with the vector coming from the parametrization of the curve.
However this is not always true as it is illustrated by the next example.
\end{Example}

\begin{figure}
\begin{equation*}
\begin{array}{|c|r|l|r|l|}
\hline
\text{Ray} & \multicolumn{2}{c|}{\LCT} & \multicolumn{2}{c|}{\LCT^{\ast}} \\ \hline
(2,5,5) & 4/3 & 1.3333 & 4/3 & 1.3333\\
(4,10,11) & 25/19 & 1.3158 & 25/19 & 1.3158\\
(2,6,5) & 13/9 & 1.4444 & 13/9 & 1.4444\\
(2,4,5) & 11/8 & 1.375 & 11/8 & 1.375\\
(2,5,4) & 11/8 & 1.375 & 11/8 & 1.375\\
(1,3,2) & 3/2 & 1.5 & 3/2 & 1.5\\
(4,9,10) & 24/19 & 1.2631 & 23/18 & 1.2778\\
(3,7,8) & 19/15 & 1.2667 & 9/7 & 1.2857\\
\mathbf{(1,2,2)} & \mathbf{5/4} & \mathbf{1.25} & 5/4 & 1.25\\
(1,3,3) & 7/5 & 1.4 & 7/5 & 1.4\\
(2,5,6) & 13/10 & 1.3 & 13/10 & 1.3\\
(2,3,5) & 5/3 & 1.6667 & 5/3 & 1.6667\\
(2,5,7) & 7/5 & 1.4 & 7/5 & 1.4\\
(1,1,2) & 2 & 2 & 2 & 2\\
(1,3,4) & 8/5 & 1.6 & 8/5 & 1.6\\
(1,2,3) & 3/2 & 1.5 & 3/2 & 1.5\\
(1,5,0) & 2 & 2 & \infty & \infty \\
\hline
\end{array}\
\begin{array}{|c|r|l|r|l|}
\hline
\text{Ray} & \multicolumn{2}{c|}{\LCT} & \multicolumn{2}{c|}{\LCT^{\ast}} \\ \hline
(2,5,0) & 2 & 2 & \infty & \infty \\
(1,3,0) & 2 & 2 & \infty & \infty \\
(1,2,0) & 2 & 2 & \infty & \infty \\
(1,1,0) & 2 & 2 & \infty & \infty \\
(1,0,5) & 2 & 2 & \infty & \infty \\
(2,0,5) & 2 & 2 & \infty & \infty \\
(2,0,3) & 2 & 2 & \infty & \infty \\
(1,0,2) & 2 & 2 & \infty & \infty \\
(1,0,3) & 2 & 2 & \infty & \infty \\
(1,0,1) & 2 & 2 & \infty & \infty \\
(1,0,0) & 2 & 2 & \infty & \infty \\
(0,2,1) & 2 & 2 & \infty & \infty \\
(0,1,2) & 2 & 2 & \infty & \infty \\
(0,1,1) & 2 & 2 & \infty & \infty \\
(0,1,0) & 2 & 2 & \infty & \infty \\
(0,0,1) & 2 & 2 & \infty & \infty \\
\hline \end{array}
\end{equation*}
\caption{Rays for curve $t^3,t^7,t^8$}
\label{Fig378}
\end{figure}

\begin{Example}
Let $V\subset\mathbb{A}^{3}$ be the monomial curve given by the parametrization $(t^3,t^7,t^8)$. The ideal $\mathfrak{a}$ defining $V$ is generated by binomials and the standard basis is:
\begin{equation*}
\mathfrak{a}=
\left\langle
x_{1}^{2}x_{3}-x_{2}^{2},
x_{1}x_{2}^{3}-x_{3}^{3},
x_{1}^{3}x_{2}-x_{3}^{2},
x_{2}^{5}-x_{1}x_{3}^{4},
x_{1}^{5}-x_{2}x_{3}
\right\rangle.
\end{equation*}
For computations it is enough to consider a set of generators, say,
\begin{equation*}
\mathfrak{a}=
\left\langle
x_{1}^{2}x_{3}-x_{2}^{2},
x_{1}^{3}x_{2}-x_{3}^{2},
x_{1}^{5}-x_{2}x_{3}
\right\rangle.
\end{equation*}
The triple $(M^{+}$, $M^{-},u)$ associated to these generators is
\begin{equation*}
M^{+}=
\left(
\begin{array}{*{3}{r}}
2 & 0 & 1\\
3 & 1 & 0\\
5 & 0 & 0
\end{array}
\right)
\qquad
M^{-}=
\left(
\begin{array}{*{3}{r}}
0 & 2 & 0\\
0 & 0 & 2\\
0 & 1 & 1
\end{array}
\right)
\qquad
u=(1,1,1).
\end{equation*}

The transpose of $A$, after deleting superfluous rows, is
\begin{equation*}
A^t=
\left(
\begin{array}{*{14}{c}}
1 & 0 & 0 & -1 & -3 & 2 & 2 & 2 & 3 & 3 & 5 & 5 & 5 & 0 \\
0 & 1 & 0 & -1 & 0 & -2 & 0 & -1 & -1 & 1 & -2 & 0 & -1 & 2 \\
0 & 0 & 1 & 1 & 1 & 1 & -1 & 0 & 0 & -2 & 0 & -2 & -1 & -2
\end{array}
\right).
\end{equation*}
We obtain $33$ rays and the minimum of $\LCT(M^{+},M^{-},u,0)$ is $\lct(\mathfrak{a})=5/4$.
Figure~\ref{Fig378} shows the list of all the rays.
Note that the minimum is achieved at ray $(1,2,2)$ ant not at ray $(3,7,8)$.
\end{Example}

\begin{Example}
Let $V\subset\mathbb{A}^{4}$ be the monomial curve given by the parametrization $(t^5,t^6,t^8,t^9)$. The ideal $\mathfrak{a}$ defining $V$ is generated by six binomials
\begin{equation*}
\mathfrak{a}=\langle
x_{2}x_{3}-x_{1}x_{4},
x_{1}^{2}x_{3}-x_{4}^{2},
x_{2}^{3}-x_{1}^{2}x_{3},
x_{1}x_{2}^{2}-x_{3}x_{4},
x_{1}^{2}x_{2}-x_{3}^{2},
x_{1}^{3}-x_{2}x_{4}
\rangle.
\end{equation*}
The triple $(M^{+},M^{-},u)$ is
\begin{equation*}
M^{+}=
\left(
\begin{array}{*{4}{r}}
0 & 1 & 1 & 0\\
2 & 0 & 1 & 0\\
0 & 3 & 0 & 0\\
1 & 2 & 0 & 0\\
2 & 1 & 0 & 0\\
3 & 0 & 0 & 0
\end{array}
\right),
\qquad
M^{-}=
\left(
\begin{array}{*{4}{r}}
1 & 0 & 0 & 1\\
0 & 0 & 0 & 2\\
2 & 0 & 1 & 0\\
0 & 0 & 1 & 1\\
0 & 0 & 2 & 0\\
0 & 1 & 0 & 1
\end{array}
\right)
\qquad
u=(1,1,1,1,1,1).
\end{equation*}
The matrix $A$, after deleting superfluous rows, has $41$ rows, we obtain a list of $848$ rays.
The minimum of $\LCT(M^{+},M^{-},u,0)$ is $23/12=\lct(W,\mathfrak{a})$ which is attained at the ray $(4,5,6,7)$.
The function $\LCT^{\ast}(M^{+},M^{-},u,0)$ has minimum equal to $2$, which corresponds to the log-canonical threshold of the monomial ideal
\begin{equation*}
\langle
x_{2}x_{3},x_{1}x_{4},
x_{1}^{2}x_{3},x_{4}^{2},
x_{2}^{3}, 
x_{1}x_{2}^{2},x_{3}x_{4},
x_{1}^{2}x_{2},x_{3}^{2},
x_{1}^{3},x_{2}x_{4}
\rangle.
\end{equation*}
\end{Example}

\begin{Example}
Let $V\subset\mathbb{A}^{4}$ be the toric surface given by the parametrization
$$(t_{1}t_{2}^{3},t_{1}^{2}t_{2}^{2},t_{1}^{3}t_{2}^{2},t_{1}t_{2}^{7}).$$
The ideal $\mathfrak{a}$ defining $V$ is generated by three binomials
\begin{equation*}
\mathfrak{a}=\langle
x_{2}x_{4}-x_{1}^{3},
x_{3}^{4}x_{4}-x_{1}x_{2}^{6},
x_{1}^{2}x_{3}^{4}-x_{2}^{7}
\rangle.
\end{equation*}
The triple $(M^{+},M^{-},u)$ is
\begin{equation*}
M^{+}=
\left(
\begin{array}{*{4}{r}}
0 & 1 & 0 & 1\\
0 & 0 & 4 & 1\\
2 & 0 & 4 & 0
\end{array}
\right),
\qquad
M^{-}=
\left(
\begin{array}{*{4}{r}}
3 & 0 & 0 & 0\\
1 & 6 & 0 & 0\\
0 & 7 & 0 & 0
\end{array}
\right)
\qquad
u=(1,1,1).
\end{equation*}
The matrix $A$, after deleting superfluous rows, has $19$ rows and we obtain a list of $124$ rays.
The minimum of $\LCT(M^{+},M^{-},u,0)$ is $99/76=\lct(W,\mathfrak{a})$ which is attained at the ray $(4,12,19,0)$.
The function $\LCT^{\ast}(M^{+},M^{-},u,0)$ has minimum equal to $17/12$, which corresponds to the log-canonical threshold of the monomial ideal
\begin{equation*}
\langle
x_{2}x_{4}, x_{1}^{3},
x_{3}^{4}x_{4}, x_{1}x_{2}^{6},
x_{1}^{2}x_{3}^{4}, x_{2}^{7}
\rangle.
\end{equation*}
\end{Example}

\begin{Example} \label{ExDepCoef}
Set $\mathfrak{a}_1$ the ideal in $k[x_1,x_2,x_3,x_4,x_5]$
\begin{multline*}
\mathfrak{a}_1=\left\langle
x_{2}^{2}x_{4}-x_{1}x_{3}x_{4},
x_{2}^{2}x_{5}-x_{1}x_{3}x_{5},
x_{1}^{3}x_{4}-x_{3}^{2}x_{4},
x_{1}^{3}x_{5}-x_{3}^{2}x_{5}
\right\rangle= \\
=\left\langle
x_{2}^{2}-x_{1}x_{3},
x_{1}^{3}-x_{3}^{2}
\right\rangle
\left\langle x_4,x_5\right\rangle.
\end{multline*}
Note that $\left\langle
x_{2}^{2}-x_{1}x_{3},
x_{1}^{3}-x_{3}^{2}
\right\rangle$ is the ideal of the monomial curve $(t^4,t^5,t^6)$ in $\mathbb{R}^3$.

Set $\mathfrak{a}_2$ to be the ideal
\begin{equation*}
\mathfrak{a}_2=\left\langle
x_{2}^{2}x_{4}-x_{1}x_{3}x_{4},
x_{2}^{2}x_{5}+x_{1}x_{3}x_{5},
x_{1}^{3}x_{4}-x_{3}^{2}x_{4},
x_{1}^{3}x_{5}-x_{3}^{2}x_{5}
\right\rangle.
\end{equation*}
Note that the triples associated to $\mathfrak{a}_1$ and $\mathfrak{a}_2$ are $(M^{+},M^{-},u_1)$ and $(M^{+},M^{-},u_2)$ respectively, where
\begin{equation*}
M^{+}=
\left(
\begin{array}{*{5}{r}}
0 & 2 & 0 & 1 & 0\\
0 & 2 & 0 & 0 & 1\\
3 & 0 & 0 & 1 & 0\\
3 & 0 & 0 & 0 & 1
\end{array}
\right),
\qquad
M^{-}=
\left(
\begin{array}{*{5}{r}}
1 & 0 & 1 & 1 & 0\\
1 & 0 & 1 & 0 & 1\\
0 & 0 & 2 & 1 & 0\\
0 & 0 & 2 & 0 & 1
\end{array}
\right),
\qquad
\begin{array}{l}
u_1=(1,1,1,1), \\
u_2=(1,-1,1,1).
\end{array}
\end{equation*}
The fan $\Gamma$, which is the same for both ideals, has 177 rays.
The minimum of the function $\LCT(M^{+},M^{-},u_1,0)$ is $17/12$ which is attained at $(4,5,6,0,0)$.
The minimum of $\LCT(M^{+},M^{-},u_2,0)$ is $3/2$ which is attained at $(4,5,6,0,0)$, it is also attained at $(1,1,1,0,0)$ and at $(2,3,4,0,0)$.

The ideals $\mathfrak{a}_1$ and $\mathfrak{a}_2$ only differ in one coefficient and they have different log canonical threshold:
\begin{equation*}
\lct(W,\mathfrak{a}_1)=\dfrac{17}{12},\qquad
\lct(W,\mathfrak{a}_2)=\dfrac{3}{2}.
\end{equation*}
This is an example that illustrates that the log canonical threshold of a binomial ideal depends not only on the exponents of the monomials, but also on the coefficients.

Set $v=(4,5,6,0,0)$, in what follows we are giving the details of the computations of 
$\LCT(M^{+},M^{-},u_1,0)(4,5,6,0,0)$ and $\LCT(M^{+},M^{-},u_2,0)(4,5,6,0,0)$.
Note that $\alpha=M^{+}v=M^{-}v=(10,10,12,12)$, $\beta=(M^{+}-M^{-})v=(0,0,0,0)$. The permutation $\varepsilon$ is such that $\varepsilon(1)=1$, $\varepsilon(2)=2$, $\varepsilon(3)=3$, $\varepsilon(4)=4$. The rank of $M^{+}-M^{-}$ is $s=2$.

For the ideal $\mathfrak{a}_1$, we have that $\bar{r}=2$, since $\beta_1,\beta_2,\beta_3,\beta_4$ and 
$1,1,1,1$ are compatible. 
The sequence of Remark~\ref{RemDefni} is $n_1=1$, $n_2=3$.

The value $\LCT(M^{+},M^{-},u_1,0)(4,5,6,0,0)$ is the minimum of
\begin{gather*}
\bar{r}=s=2, \\
\dfrac{|v|}{\alpha_1}=\dfrac{15}{10}, \\
\dfrac{|v|+(\alpha_3-\alpha_1)}{\alpha_3}=\dfrac{17}{12}.
\end{gather*}
For the ideal $\mathfrak{a}_2$, we have that $\bar{r}=1$ since $\beta_1,\beta_2$ and $1,-1$ are not compatible (\ref{DefBetaUComp})
The sequence of Remark~\ref{RemDefni} is now $n_1=1$, $n_2=2$.

The value $\LCT(M^{+},M^{-},u_2,0)(4,5,6,0,0)$ is the minimum of
\begin{gather*}
\dfrac{|v|}{\alpha_1}=\dfrac{15}{10}, \\
\dfrac{|v|+(\alpha_2-\alpha_1)}{\alpha_2}=\dfrac{15}{10}.
\end{gather*}
\end{Example}

\begin{Example}
Example~\ref{ExDepCoef} may be generalized to the ideal
\begin{equation*}
\mathfrak{a}_3=\left\langle
x_{2}^{2}x_{4}-u_1x_{1}x_{3}x_{4},
x_{2}^{2}x_{5}-u_3x_{1}x_{3}x_{5},
x_{1}^{3}x_{4}-u_2x_{3}^{2}x_{4},
x_{1}^{3}x_{5}-u_4x_{3}^{2}x_{5}
\right\rangle,
\end{equation*}
where $u_1,u_2,u_3,u_4\in k^{\ast}$.
In this case we have that
\begin{equation*}
\lct(\mathbb{A}^5,\mathfrak{a}_3)=
\left\{
\begin{array}{lll}
3/2 & \text{if} & u_1\neq u_2 \\ 
17/12 & \text{if} & u_1=u_2
\end{array} 
\right.
\end{equation*}

\end{Example}


\begin{thebibliography}{DGPS12}

\bibitem[AF92]{AvisFukuda1992}
David Avis and Komei Fukuda.
\newblock A pivoting algorithm for convex hulls and vertex enumeration of
  arrangements and polyhedra.
\newblock {\em Discrete Comput. Geom.}, 8(3):295--313, 1992.
\newblock ACM Symposium on Computational Geometry (North Conway, NH, 1991).

\bibitem[BEV05]{BravoEncinasVillamayor2005}
{Ana Mar{\'{\i}}a} Bravo, Santiago Encinas, and Orlando Villamayor.
\newblock A simplified proof of desingularization and applications.
\newblock {\em Rev. Mat. Iberoamericana}, 21(2):349--458, 2005.

\bibitem[BL04]{BlickleLazarsfeld2004}
Manuel Blickle and Robert Lazarsfeld.
\newblock An informal introduction to multiplier ideals.
\newblock In {\em Trends in commutative algebra}, volume~51 of {\em Math. Sci.
  Res. Inst. Publ.}, pages 87--114. Cambridge Univ. Press, Cambridge, 2004.

\bibitem[Bli04]{Blickle2004}
Manuel Blickle.
\newblock Multiplier ideals and modules on toric varieties.
\newblock {\em Math. Z.}, 248(1):113--121, 2004.

\bibitem[CLS11]{CoxLittleSchenck2011}
David~A. Cox, John~B. Little, and Henry~K. Schenck.
\newblock {\em Toric varieties}, volume 124 of {\em Graduate Studies in
  Mathematics}.
\newblock American Mathematical Society, Providence, {RI}, 2011.

\bibitem[DGPS12]{Singular2012}
Wolfram Decker, Gert-Martin Greuel, Gerhard Pfister, and Hans Sch\"onemann.
\newblock {\sc Singular} {3-1-6} --- {A} computer algebra system for polynomial
  computations.
\newblock \url{http://www.singular.uni-kl.de}, 2012.

\bibitem[ELSV04]{EinLazarsfeldSmithVarolin2004}
Lawrence Ein, Robert Lazarsfeld, Karen~E. Smith, and Dror Varolin.
\newblock Jumping coefficients of multiplier ideals.
\newblock {\em Duke Math. J.}, 123(3):469--506, 2004.

\bibitem[ES96]{EisenbudSturmfels1996}
David Eisenbud and Bernd Sturmfels.
\newblock Binomial ideals.
\newblock {\em Duke Math. J.}, 84(1):1--45, 1996.

\bibitem[EV00]{EncinasVillamayor2000}
Santiago Encinas and Orlando Villamayor.
\newblock A course on constructive desingularization and equivariance.
\newblock In {\em Resolution of singularities (Obergurgl, 1997)}, volume 181 of
  {\em Progr. Math.}, pages 147--227. Birkh\"auser, Basel, 2000.

\bibitem[Ful93]{Fulton1993}
William Fulton.
\newblock {\em Introduction to toric varieties}, volume 131 of {\em Annals of
  Mathematics Studies}.
\newblock Princeton University Press, Princeton, {NJ}, 1993.
\newblock The William H. Roever Lectures in Geometry.

\bibitem[Gow05]{Goward2005}
Russell~A. Goward, Jr.
\newblock A simple algorithm for principalization of monomial ideals.
\newblock {\em Trans. Amer. Math. Soc.}, 357(12):4805--4812 (electronic), 2005.

\bibitem[GPT02]{GonzalezTeissier2002}
Pedro~Daniel Gonz{\'a}lez~P{\'e}rez and Bernard Teissier.
\newblock Embedded resolutions of non necessarily normal affine toric
  varieties.
\newblock {\em C. R. Math. Acad. Sci. Paris}, 334(5):379--382, 2002.

\bibitem[GPT14]{GonzalezTeissier2014}
Pedro~D. Gonz{\'a}lez~P{\'e}rez and Bernard Teissier.
\newblock Toric geometry and the {S}emple--{N}ash modification.
\newblock {\em Rev. R. Acad. Cienc. Exactas F\'\i s. Nat. Ser. A Math. RACSAM},
  108(1):1--48, 2014.

\bibitem[Hau03]{Hauser2003}
Herwig Hauser.
\newblock The {H}ironaka theorem on resolution of singularities (or: {A} proof
  we always wanted to understand).
\newblock {\em Bull. Amer. Math. Soc. (N.S.)}, 40(3):323--403 (electronic),
  2003.

\bibitem[How01]{Howald2001}
J.~A. Howald.
\newblock Multiplier ideals of monomial ideals.
\newblock {\em Trans. Amer. Math. Soc.}, 353(7):2665--2671 (electronic), 2001.

\bibitem[Laz04]{Lazarsfeld2004_2}
Robert Lazarsfeld.
\newblock {\em Positivity in algebraic geometry. {II}}, volume~49 of {\em
  Ergebnisse der Mathematik und ihrer Grenzgebiete. 3. Folge. A Series of
  Modern Surveys in Mathematics [Results in Mathematics and Related Areas. 3rd
  Series. A Series of Modern Surveys in Mathematics]}.
\newblock Springer-Verlag, Berlin, 2004.
\newblock Positivity for vector bundles, and multiplier ideals.

\bibitem[Laz10]{Lazarsfeld2010}
Robert Lazarsfeld.
\newblock A short course on multiplier ideals.
\newblock In {\em Analytic and algebraic geometry}, volume~17 of {\em IAS/Park
  City Math. Ser.}, pages 451--494. Amer. Math. Soc., Providence, RI, 2010.

\bibitem[Oda88]{Oda1988}
Tadao Oda.
\newblock {\em Convex bodies and algebraic geometry}, volume~15 of {\em
  Er\-gebnisse der Mathematik und ihrer Grenzgebiete (3) [Results in
  Mathematics and Related Areas (3)]}.
\newblock Springer-Verlag, Berlin, 1988.
\newblock An introduction to the theory of toric varieties, Translated from the
  Japanese.

\bibitem[ST09]{ShibutaTakagi2009}
Takafumi Shibuta and Shunsuke Takagi.
\newblock Log canonical thresholds of binomial ideals.
\newblock {\em Manuscripta Math.}, 130(1):45--61, 2009.

\bibitem[Tei04]{Teissier2004}
Bernard Teissier.
\newblock Monomial ideals, binomial ideals, polynomial ideals.
\newblock In {\em Trends in commutative algebra}, volume~51 of {\em Math. Sci.
  Res. Inst. Publ.}, pages 211--246. Cambridge Univ. Press, Cambridge, 2004.

\bibitem[Tho14]{Thompson2014}
Howard~M. Thompson.
\newblock Multiplier ideals of monomial space curves.
\newblock {\em Proc. Amer. Math. Soc. Ser. B}, 1:33--41, 2014.

\bibitem[Zei06]{Zeillinger2006}
Dominik Zeillinger.
\newblock A short solution to {H}ironaka's polyhedra game.
\newblock {\em Enseign. Math. (2)}, 52(1-2):143--158, 2006.

\end{thebibliography}
\end{document}